\documentclass[11pt,reqno,twoside]{article}

\usepackage{graphicx}
\usepackage{float}
\usepackage{subfigure}
\usepackage{amssymb}
\usepackage{epstopdf}
\usepackage[asymmetric,top=3.5cm,bottom=4.3cm,left=3.1cm,right=3.1cm]{geometry}
\usepackage{hyperref}
\usepackage{bm}
\usepackage{amsmath,amsfonts,amsthm,mathrsfs,amssymb,cite}
\usepackage[usenames]{color}
\usepackage{array} 
\usepackage{paralist} 
\usepackage{verbatim} 
\usepackage{tabularx}
\usepackage{fancyhdr}
\usepackage{graphicx, subfigure,threeparttable}

\newtheorem{thm}{Theorem}[section]

\newtheorem{lem}{Lemma}[section]
\newtheorem{prop}{Proposition}[section]
\newtheorem{asum}{Assumption}[section]
\theoremstyle{definition}
\newtheorem{defn}{Definition}[section]
\theoremstyle{remark}

\newtheorem{rem}{Remark}[section]
\numberwithin{equation}{section}

\newcommand{\norm}[1]{\left\Vert#1\right\Vert}
\newcommand\normx[1]{\Vert#1\Vert}

\newcommand{\bu}{\mathbf{u}}

\newcommand{\bx}{\mathbf{x}}
\newcommand{\bnu}{\bm{\nu}}

\newcommand{\ba}{\mathbf{a}}
\newcommand{\bb}{\mathbf{b}}
\newcommand{\bg}{\mathbf{g}}
\newcommand{\bw}{\mathbf{w}}
\newcommand{\bv}{\mathbf{v}}

\newcommand{\ri}{\mathrm{i}}

\newcommand{\bGa}{\mathbf{\Gamma}}
\newcommand{\bvarphi}{\bm{\varphi}}
\newcommand{\bpsi}{\bm{\psi}}
\newcommand{\by}{\mathbf{y}}
\newcommand{\bz}{\mathbf{z}}
\newcommand{\bS}{\mathbf{S}}
\newcommand{\bK}{\mathbf{K}}
\newcommand{\bI}{\mathbf{I}}
\newcommand{\bxi}{\bm{\xi}}
\newcommand{\bzeta}{\bm{\zeta}}

\newcommand{\bbe}{\bm{\beta}}

\newcommand{\Acal}{\mathcal{A}}
\newcommand{\Bcal}{\mathcal{B}}
\newcommand{\Ccal}{\mathcal{C}}
\newcommand{\Ecal}{\mathcal{E}}

\newcommand{\Lcal}{\mathcal{L}}
\newcommand{\Jcal}{\mathcal{J}}

\newcommand{\Hcal}{\mathcal{H}}

\newcommand{\Ncal}{\mathcal{N}}
\newcommand{\Ocal}{\mathcal{O}}

\newcommand{\Qcal}{\mathcal{Q}}

\title{Resonant modes of two hard inclusions within a soft elastic material and their stress estimate }

\begin{document}
\author{
	Hongjie Li\footnote{Yau Mathematical Sciences Center, Tsinghua University, Beijing, China. The work of this author was substantially supported by Research Start Fund (53331004324). (hongjieli@tsinghua.edu.cn; hongjie\_li@yeah.net).}
	\and
	Longjuan Xu\footnote{Academy for Multidisciplinary Studies, Capital Normal University, Beijing 100048, China.
			The work of this author was partially supported by NSF of China (12301141), and Beijing Municipal Education Commission Science and Technology Project (KM202410028001). (longjuanxu@cnu.edu.cn).}
}
\date{}
\maketitle

\begin{abstract}

In this paper, we are concerned with subwavelength resonant modes of two hard inclusions embedding in soft elastic materials to realize negative materials in elasticity. All the $12$ subwavelength resonant frequencies are derived explicitly for general convex resonators. In addition, the resonant modes are categorized into dipolar, quadrupolar, and hybrid groups, facilitating the effective realization of negative mass density, negative shear modulus and double-negative properties (both mass density and shear modulus)  in elastic metamaterials. Moreover, we analyze the stress distribution between two hard inclusions when they are closely touching. We also precisely derive the sharp blow-up rates of the gradient estimates of the resonant modes. Our findings show that certain resonant modes have bounded stress estimates when the curvature of the hard inclusions is appropriately designed.
These results provide valuable insights into the stability of the design and fabrication of elastic metamaterials.
Lastly, we express the scattering wave fields explicitly in terms of resonant modes, offering a clear understanding of their impact on the overall scattering behavior. 

\medskip 



\end{abstract}

\section{Introduction}
Metamaterials is a class of man-made materials that exhibit exotic and distinct properties not found in nature. In history, the exploration of artificial materials began in the late 19th century. The first theoretical investigation of the negative-index materials (NIMs) was given by Victor Veselago in 1967 by assuming negative values of permittivity and permeability \cite{vese3699}, which opened up new possibilities in material science. However, NIMs did not receive much attention at that time, as many researchers believed they did not exist.
It was not until the early 21st century that the first NIM was realized in the laboratory using artificial lumped-element loaded transmission lines \cite{sss8847}. Since then, metamaterials have attracted significant attention from both industrial and academic communities,  leading to extensive research and innovation in the field. The peculiar physical properties of metamaterials have led to numerous important applications. Invisibility cloaking, for instance, has been a notable area of interest, enabling the manipulation of electromagnetic waves to render objects invisible \cite{ACK6055, LL0165, LL9023, LLL6014, LL6090, MN1715}. Super-resolution imaging is another application where metamaterials have shown promise, allowing for imaging beyond the diffraction limit \cite{CLH1677, BLLW9091, LLLW9004}. Metamaterials also find applications in waveguides, where they can control and guide the propagation of waves \cite{AE1274, AHY4171, LP5021}. Additionally, commercially available antennas have benefited from the use of metamaterials, enhancing their performance and functionality \cite{KKC1423, QCJ2662}.

Metamaterials are indeed synthetic composite materials created by arranging sub-wavelength resonators, also known as ``meta-atoms", in periodic or random patterns. These meta-atoms are typically constructed by embedding elements like plastics and metals into a matrix, where these materials enjoy high contrasts in specific physical parameters. The term ``sub-wavelength" indicates that the size of the meta-atom is smaller than the wavelength of the incident wave, while ``resonator" denotes that resonance occurs within a specific frequency range. Metamaterials can be classified into various types, including acoustic, electromagnetic, and elastic metamaterials, based on their modified properties. Acoustic metamaterials often utilize sub-wavelength resonators created by embedding bubbles in liquids \cite{AFGL007, CP7599, PhReE.7005}. In electromagnetic metamaterials, common meta-atoms include silicon nanoparticles, split-ring resonators, and conducting wires \cite{ALLZ3224, SC4846, KML2472}. 

In this paper, we mainly consider elastic metamaterials, which have been less explored compared to acoustic and electromagnetic metamaterials.
This limited exploration is mainly due to the coupling between longitudinal and transverse waves in elasticity, causing elastic waves to exhibit properties of both acoustic and electromagnetic waves.
Elastic materials are characterized by three physical parameters: bulk modulus, mass density, and shear modulus. 
To achieve negative values for these parameters, monopolar, dipolar, and quadrupolar resonances should be induced, respectively \cite{NaMLai2011}.
In particular, the works \cite{CTL1646, LSM0301} use the structure of bubbles embedded in soft elastic materials (BSE structures) to induce monopolar resonance, resulting in negative bulk modulus.
The paper \cite{LLZ0572} provides the first mathematical proof of the resonances in BSE structures. To realize negative mass density, studies \cite{Liu00scie, PhReB7101} have investigated hard inclusions embedded in soft elastic materials (HISE structures) to induce dipolar resonance.
The corresponding mathematical study is presented in \cite{LZ2861}.

This paper concentrates on the investigation of a dimer consisting of two hard inclusions embedded in soft elastic materials (THES structures) to induce quadrupolar resonance, thereby achieving negative shear modulus values. For more physical studies, we refer to \cite{NaMLai2011}. It has been demonstrated that both the electric field and the stress field exhibit blow-up behavior between adjacent inclusions when they are brought closer. Please refer to the studies in the context of electrostatics and electromagnetic \cite{ABV2015,AKLLL2007,BLY2009,dfl2022,DL2019,DLY2022,dlz2024,G2015,KLY2015,KL2019,LY2023,LY2009,MPM1988,P1989,W2023} (this is far from a complete list), and stress field within the framework of linear elasticity \cite{ACKLY,BLL2015,BLL2017,KY2019,L2021,LY2017,lz2020}. In the current study, we explore the blow-up behavior of eigenmode gradients as the resonators approach each other. We would like to mention that the blow-up phenomenon here is different from the aforementioned studies since we are considering the resonant scenario. Related studies on the Helmholtz problem can be found in  \cite{ADY2020,LZ2023}.

The contributions of this study are as follows.
First, all 12 resonant frequencies of the THES structure are explicitly derived in Theorem \ref{thm:renfre} under certain assumptions about the dimer structure.
These resonant frequencies exhibit different asymptotic behaviors based on the separation distance of the inclusions and the contrast in material parameters (see Theorem \ref{thm:renfre} and Remark \ref{rem:disren}).  
It is noted that deriving the resonant frequencies involves analyzing a $12\times 12$ matrix, making the analysis highly nontrivial.
Moreover, the resonant modes are classified into three different groups. The first one belongs to the dipolar resonance, which can be used to realize the negative mass density in elastic metamaterials. The second one belongs to the quadrupolar resonances, which can be utilized to achieve negative shear modulus. The last one includes both dipolar and quadrupolar resonances, thus helping to realize the negative mass density and shear modulus simultaneously. Please refer to Remark \ref{rmk-res} for more details. 
Second, we analyze the stress distribution between two closely touching hard inclusions under the subwavelength resonance.
The sharp blow-up rates of the gradient estimate for the resonant modes are derived exactly in Theorem \ref{thm:blowrate}. 
It is shown that the stress concentration of some resonant modes is very strong, which leads to poor stability of the metamaterials. However, the stress estimates of some other resonant modes are of $\Ocal(1)$ if the curvature of the hard inclusions is properly chosen. These results provide valuable insights into the stability of the design and fabrication of elastic metamaterials (see Remark \ref{rem:debr}).
 Third, the scattering wave fields are explicitly expressed in terms of the resonant modes in Theorem \ref{thm:scat}, offering a clear understanding of their impact on the overall scattering behavior of the THES structure.

The paper is organized as follows. Section 2 presents the mathematical formulation of the problem. In Section 3, we provide some preliminaries and auxiliary results for the subsequent analysis.
In Section 4, we explicitly derive the subwavelength resonant frequencies and investigate their resonant modes. 
Section 5 analyzes the blow-up behavior of stress fields between two hard inclusions. Section 6 derives the scattering field in the quasi-static regime using the resonant modes. 
In Appendix \ref{Appendix}, we provide the proofs of some auxiliary results in the text. 


\section{Mathematical formulation}
In this section, we present the mathematical formulation of the problem considered in the paper. 
We study the configuration of a dimer consisting of two hard elastic inclusions embedded in soft elastic materials in three dimensions.
The matrix material is described by two Lam\'e parameters $(\lambda, \mu)$, satisfying the strong ellipticity conditions
\begin{equation}\label{eq:painma}
    \mu>0\quad  \text{and}\quad3\lambda+2\mu>0.
\end{equation}
The density in the background is designated by $\rho$.
Let $D_{1}$ and $D_{2}$ be two disjoint convex open subsets in $\mathbb R^3$ with $C^{2,\alpha}$ boundaries, $\alpha\in(0,1)$, and 
denote the two hard inclusions. The corresponding Lam\'e parameters and the density of the hard inclusions are parameterized by $(\tilde{\lambda}, \tilde{\mu})$ and $\tilde{\rho}$, respectively. Then, we introduce the three dimensionless contrast parameters $\delta$, $\eta$ and $\tau$ defined by 
\begin{equation}\label{eq:conpara}
(\tilde{\lambda}, \tilde{\mu}) = \frac{1}{\delta}(\lambda, \mu), \quad \tilde{\rho} = \frac{1}{\eta}\rho, \quad \tau=\sqrt{\delta/\eta}.
\end{equation}
Since we are considering the configuration that hard inclusions are embedded within soft elastic materials (e.g. lead inclusions coated with silicone in \cite{Liu00scie}), the contrast parameters $\delta$, $\eta$ and $\tau$ satisfy the following conditions:
\[
 \delta\ll 1, \quad \eta\ll 1, \quad \mbox{and}\quad \tau\leq \Ocal(1).
\]

Define the domain $D$ by $D=D_1\cup D_2$ and denote by $D^{e}:=\mathbb{R}^{3} \backslash \overline{D}$ the exterior of the domain $D$. 
Define the elasticity tensor $\mathbb{C}=(C_{ijkl})_{i,j,k,l=1}^3$ by 
$$C_{ijkl}=\lambda\delta_{ij}\delta_{kl} +\mu(\delta_{ik}\delta_{jl}+\delta_{il}\delta_{jk}).$$
Let $\bu^i$ be a time-harmonic incident elastic wave satisfying the elastic equation in the entire space $\mathbb{R}^3$
\begin{equation}\label{eq:inci}
\mathcal{L}_{ {\lambda}, {\mu}}\bu^i(\bx) + {\rho}\omega^2\bu^i(\bx) =0,
\end{equation}
where $\omega>0$ denotes the angular  frequency. 
In \eqref{eq:inci}, the Lam\'e operator $ \mathcal{L}_{\lambda, \mu}$ associated with the parameters $(\lambda,\mu)$ is defined by 
\begin{equation}\label{op:lame}
 \Lcal_{\lambda,\mu}\bu^i:=\nabla\cdot(\mathbb{C}e(\bu^i))=\mu \triangle\bu^i + (\lambda+ \mu)\nabla\nabla\cdot\bu^i.
\end{equation}
Here, $e(\bu^i)$ is the strain tensor given by
$$e(\bu^i)=\frac{1}{2}(\nabla \bu^i+(\nabla\bu^i)^{t}),$$
where $t$ signifies the transpose. Then the total displacement field ${\mathbf{u}}$ of the above described system satisfies the following partial differential equations (PDEs) \cite{L0069}
\begin{equation}\label{eq:or2}
	\left\{
	\begin{array}{ll}
		\mathcal{L}_{ {\lambda},  {\mu}} \bu(\bx) + \rho\tau^2\omega^2 \bu(\bx) =0,    & \bx\in D \medskip \\
		\mathcal{L}_{\lambda, \mu} \bu(\bx) + \rho\omega^2 \bu(\bx) = 0, & \bx\in D^{e},\medskip \\
		\bu(\bx)|_- =  \bu(\bx)|_+,      & \bx\in\partial D, \medskip \\
		\partial_{ {\bnu}} \bu(\bx)|_- = \delta\partial_{\bnu} \bu(\bx)|_+, &  \bx\in\partial D,
	\end{array}
	\right.
\end{equation}
where $\tau$ is given in \eqref{eq:conpara} and the subscripts $\pm$ indicate the limits from outside and inside of $D$, respectively.
Here, the traction operator $\partial_{\bnu}$ is defined by
\begin{equation*}
	\partial_{ {\bnu}}\bu(\bx)  : =\lambda(\nabla \cdot \bu) \bnu +\mu\left(\nabla \bu+(\nabla \bu)^{t}\right) \bnu,
\end{equation*}
with $\bnu$ denoting the exterior unit normal vector to $\partial D$. 
In \eqref{eq:or2}, the scattering wave $\bu^s = \bu-\bu^i$ satisfies the following radiation condition \cite{KupTPET, LLL9974}:
\begin{equation*}
	\begin{split}
		(\nabla\times\nabla\times \bu^s)(\bx)\times\frac{\bx}{|\bx|}-\mathrm{i} {k}_s\nabla\times \bu^s(\bx)=&\mathcal{O}(|\bx|^{-2}),\\
		\frac{\bx}{|\bx|}\cdot[\nabla(\nabla\cdot \bu^s)](\bx)-\mathrm{i} {k}_p\nabla \bu^s(\bx)=&\mathcal{O}(|\bx|^{-2}),
	\end{split}
\end{equation*}
as $|\mathbf{x}|\rightarrow+\infty$, where $\ri$ signifies the imaginary unit and
\begin{equation}\label{pa:ksp}
	{k}_s=\omega/{c}_s, \quad  \quad {k}_p=\omega/{c}_p,
\end{equation}
with
\begin{equation}\label{cs-cp}
{c}_s = \sqrt{{\mu}/{\rho}}, \quad \quad {c}_p=\sqrt{ ({\lambda} + 2 {\mu})/{\rho}}.
\end{equation}
It is well-known that the elastic wave can be decomposed into the shear wave (s-wave) and the compressional wave (p-wave). In \eqref{pa:ksp}, the parameters $k_s$ and $k_p$ signify the wavenumbers of the s-wave and p-wave, respectively. 

In this work, we mainly apply the potential theory to analyze the system \eqref{eq:or2}. Thus, we first present the potential theory in elasticity. Let the function $\bGa^{\omega}=(\Gamma^{\omega}_{i,j})_{i,j=1}^3$ signify the 
fundamental solution of the operator $\Lcal_{\lambda,\mu} + \rho\omega^2$ in three dimensions given by \cite{DLL7678}
\begin{equation}\label{eq:ef}
	\Gamma^{\omega}_{i,j}(\bx)=-\frac{\delta_{ij}}{4\pi\mu|\bx|}e^{\ri k_s |\bx|} + \frac{1}{4\pi \omega^2\rho}\partial_i\partial_j\frac{e^{\ri k_p|\bx|} - e^{\ri k_s|\bx|}}{|\bx|},
\end{equation}
where $\delta_{ij}$ is Kronecker delta function. In particular, when $\omega=0$, the function $\bGa^{0}$ is called the Kelvin matrix of the fundamental solution, and we denote $\bGa^{0}$ by $\bGa$ for simplicity. The function $\bGa$ has the following expression \cite{DLL262}
\[
\Gamma_{i,j} (\bx)= -\frac{1}{8\pi}\left( \frac{1}{\mu} + \frac{1}{\lambda + 2\mu} \right)  \frac{\delta_{ij}}{|\bx|}  -\frac{1}{8\pi}\left( \frac{1}{\mu} - \frac{1}{\lambda + 2\mu} \right)  \frac{\bx_i \bx_j}{|\bx|^3 }. 
\]
Then the single-layer potential associated with the fundamental solution $\bGa^{\omega}$ is defined by
\begin{equation}\label{eq:single}
	\bS_{D}^{\omega}[\bvarphi](\bx)=\int_{\partial D} \bGa^{\omega}(\bx-\by)\bvarphi(\by)ds(\by), \quad \bx\in\mathbb{R}^3,
\end{equation}
for $\bvarphi\in L^2(\partial D)^3$. On the boundary $\partial D$, the conormal derivative of the single-layer potential satisfies the following jump formula
\begin{equation}\label{eq:jump}
	\partial_{\bnu}   \bS_{ D}^{\omega}[\bvarphi]|_{\pm}(\bx)=\left( \pm\frac{1}{2}\bI +  \bK_{ D}^{\omega, *} \right)[\bvarphi](\bx) \quad \bx\in\partial D,
\end{equation}
where
\[
\bK_{ D}^{\omega, *} [\bvarphi](\bx)=\mbox{p.v.} \int_{\partial D} \partial_{\bnu_{\bx}} \bGa^{\omega}(\bx-\by)\bvarphi(\by)ds(\by),
\]
with $\mbox{p.v.}$ standing for the Cauchy principal value. It is noted that the operator $ \bK_{ D}^{\omega, *}$ in \eqref{eq:jump} is called the Neumann-Poincar\'e (N-P) operator, which is a critical operator in the analysis of metamaterials. In what follows, we denote $\bS_{ D}^{0}$, $\bK_{ D}^{0, *}$ by $\bS_{ D}$, $\bK_{ D}^{ *}$, respectively, for simplicity. More details will be given in Subsection \ref{subsec-NP}.

With the help of the potential theory presented above, the solution to the system \eqref{eq:or2} can be written as 
\begin{equation}\label{eq:sol}
	\bu=
	\left\{
	\begin{array}{ll}
		{\bS}_{ D}^{\tau\omega}[\bvarphi](\bx), & \bx\in D,  \smallskip \\
		{\bS}_{ D}^{\omega}[\bpsi](\bx) +\bu^i, &  \bx\in \mathbb{R}^3\backslash \overline{D},
	\end{array}
	\right.
\end{equation}
for some density functions $\bvarphi, \bpsi \in L^2(\partial D)$.
By matching the transmission conditions on the boundary, i.e. the third and fourth conditions in \eqref{eq:or2} and with the help of the jump formula in \eqref{eq:jump}, the density functions $\bvarphi$ and $\bpsi$ in \eqref{eq:sol} should satisfy the following system:
\begin{equation}\label{eq:or}
	\Acal(\omega,\delta) [\Phi](\bx)=F(\bx), \quad \bx\in\partial D,
\end{equation}
where
\[
\Acal(\omega,\delta)=  \left(
\begin{array}{cc}
	{\bS}_{ D}^{\tau\omega} &  -{\bS}_{ D}^{\omega}\medskip \\
	-\frac\bI{2} +  {\bK}_{ D}^{\tau\omega, *} & -\delta\left( \frac\bI{2} +  {\bK}_{ D}^{\omega, *} \right)\\
\end{array}
\right),
\;\;
\Phi= \left(
\begin{array}{c}
	\bvarphi \\
	\bpsi \\
\end{array}
\right)
\;\; \mbox{and} \;\; 
F= \left(
\begin{array}{c}
	\bu^i \\
	\delta\partial_{\bnu}  \bu^i \\
\end{array}
\right).
\]
For the subsequent discussion, we define the spaces $\Hcal:=L^2(\partial D)\times L^2(\partial D)$ and $\Hcal^1:=H^1(\partial D)\times L^2(\partial D)$. It is noted the operator $\Acal(\omega,\delta)$ is bounded from $\Hcal$ to $\Hcal^1$ (c.f. \cite{ABJH6625}). Next, we define the sub-wavelength resonance of the scattering system \eqref{eq:or2} based on the operator $\Acal(\omega,\delta)$.
\begin{defn}
	The sub-wavelength resonance of the scattering system \eqref{eq:or2} occurs if there exists a frequency $\omega\ll 1$ such that the operator $\Acal(\omega,\delta)$ has a nontrivial kernel, i.e. 
	\begin{equation}\label{eq:conre1}
		\Acal(\omega,\delta)[\Phi](\bx)=0,
	\end{equation}
	for some nontrivial $\Phi\in\Hcal$. Here, $\omega$ is called the resonant frequency (or eigenfrequency) and $\Phi$ is called the eigenvector. For each resonant frequency $\omega$, we define the corresponding resonant mode (or eigenmode) as 
 \begin{equation*}
     \bu=
	\left\{
	\begin{array}{ll}
		{\bS}_{ D}^{\tau\omega}[\bvarphi](\bx), & \bx\in D,  \smallskip \\
		{\bS}_{ D}^{\omega}[\bpsi](\bx) , &  \bx\in \mathbb{R}^3\backslash \overline{D}.
	\end{array}
	\right.
 \end{equation*}
\end{defn}

In the study, we shall systematically and comprehensively investigate the resonant characteristics of the system \eqref{eq:or2}. All the subwavelength resonant frequencies shall be derived explicitly. Then the resonant modes are classified into dipolar, quadropolar, and hybrid resonances. Furthermore, the stress fields between two hard inclusions are analyzed quantitatively. Lastly, the scattering wave fields in the system \eqref{eq:or2} are represented in terms of the subwavelength eigenmodes.

\section{Preliminaries and auxiliary results}

In this section, we present some basic notation, preliminaries and auxiliary results that will be used subsequently.

\subsection{Asymptotic notation}
We shall use the following notation in this paper. 

\begin{defn}
Let $f$ be a real- or complex-valued function, and let $g$ be a real-valued
function that remains strictly positive in a neighborhood of $\bx_0$. We write that
\begin{equation*}
f(\bx)=\Ocal(g(\bx))\quad\mbox{as}~\bx\rightarrow\bx_0
\end{equation*}
if and only if there exists a positive constant $M$ such that $|f(\bx)|\leq Mg(\bx)$ for all $\bx$ in a neighborhood of $\bx_0$.
\end{defn}

\begin{defn}
Let $f$ and $g$ be real-valued functions which are strictly positive in a neighborhood of $\bx_0$. We write that
\begin{equation*}
f(\bx)\sim g(\bx)\quad\mbox{as}~{\bx}\rightarrow\bx_0
\end{equation*}
if and only if $f(\bx)=\Ocal(g(\bx))$ as ${\bx}\rightarrow\bx_0$. 
\end{defn}

\subsection{Layer potentials and Neumann-Poincar\'e operator}\label{subsec-NP}
We first provide the asymptotic expansion of the fundamental solution $\bGa^{\omega}$ defined in \eqref{eq:ef} (c.f. \cite{LZ2861}).
\begin{lem}\label{lem:asmfun}
     For $\omega\ll 1$, there holds the following asymptotic expansion of the fundamental solution $\bGa^{\omega}$ defined in \eqref{eq:ef}
	\begin{equation*}
		\mathbf{\Gamma}^{\omega}(\bx) = \sum_{n=0}^{\infty} \omega^n  \mathbf{\Gamma}_n(\bx),
	\end{equation*}
	where
	\[
	\begin{split}
		\mathbf{\Gamma}_n(\bx) = & -\frac{\alpha_1}{4\pi} \frac{\mathrm{i}^n}{(n+2)n!}\left(\frac{n+1}{c^{n}_s} + \frac{1}{c^{n}_p} \right)  |\bx|^{n-1}\mathbf{I}  +\frac{\alpha_2}{4\pi}\frac{\mathrm{i}^n(n-1)}{(n+2)n!}\left(\frac{1}{c^{n}_s} - \frac{1}{c^{n}_p} \right)  |\bx|^{n-3}\bx\bx^t,
	\end{split}
	\]
	with 
	\[
	\alpha_1= \left( \frac{1}{\mu} + \frac{1}{2\mu +\lambda} \right), \quad  \alpha_2 = \left( \frac{1}{\mu} - \frac{1}{2\mu +\lambda} \right)
	\]
 and $c_s, c_p$ defined in \eqref{cs-cp}.
\end{lem}

From Lemma \ref{lem:asmfun}, the single layer potential $\bS_{ D}^{\omega}$ defined in \eqref{eq:single} has the following asymptotic expansion:
\begin{equation}\label{eq:sise}
	\bS_{ D}^{\omega} = \bS_{ D} + \sum_{n=1}^\infty \omega^n  \bS_{ D, n},
\end{equation}
where
\[
\bS_{ D, n}[\bvarphi](\bx)=\int_{\partial D} \bGa_n(\bx-\by)\bvarphi(\by)ds(\by).
\]
In particular, one has that
\begin{equation*}
	\bS_{ D, 1}[\bvarphi](\bx)= -\frac{\ri \alpha_1}{12\pi} \left(\frac{2}{c_s} + \frac{1}{c_p} \right) \int_{\partial D} \bvarphi(\by)ds(\by).
\end{equation*}
Then the following two lemmas hold \cite{DLL1067}.

\begin{lem}\label{lem:insin}
    The single layer potential operator $\bS_{ D}$ is invertiable from $L^2(\partial D)$ to $H^1(\partial D)$.
\end{lem}

\begin{lem}
	The norm $\normx{\bS_{ D, n}}_{\Lcal(L^2(\partial D), H^1(\partial D))} $ is uniformly bounded with respect to $n$. Moreover, the series in \eqref{eq:sise} is convergent in $\Lcal(L^2(\partial D), H^1(\partial D))$.
\end{lem}

Next, we consider the asymptotic expansion of N-P operator $ \bK_{ D}^{\omega, *}$ given in \eqref{eq:jump}. Using Lemma \ref{lem:asmfun}, we have the following asymptotic expansion:
\begin{equation}\label{eq:npse}
	\bK_{ D}^{\omega, *} = \bK_{ D}^{ *} + \sum_{n=1}^\infty \omega^n  \bK_{ D, n}^{ *},
\end{equation}
where
\[
\bK_{ D, n}^{ *} [\bvarphi](\bx)=\int_{\partial D}  \frac{\partial \bGa_n}{\partial \bnu(\bx)} (\bx-\by)\bvarphi(\by)ds(\by).
\]
In particular, we have 
\[
\bK_{ D, 1}^{ *}=0,
\]
which follows from that the function $ \mathbf{\Gamma}_1(\bx)$ is a constant. Moreover, the following lemma holds.

\begin{lem}
	The norm $\normx{  \bK_{ D, n}^{ *}  }_{\Lcal(L^2(\partial D))} $ is uniformly bounded with respect to $n$. Moreover, the series in \eqref{eq:npse} is convergent in $\Lcal(L^2(\partial D))$ (c.f. \cite{DLL1067}).
\end{lem}

Then we introduce the vector space $\mho$ that is spanned by all linear solutions to the equation
\begin{equation*}
	\left\{
	\begin{array}{ll}
		\mathcal{L}_{ {\lambda}, {\mu}}\bu=0  &  \bx\in D,  \medskip\\
		\partial_{{\bnu}}\bu=0     & \bx\in\partial D.  \medskip
	\end{array}
	\right.
\end{equation*} 
Indeed, the space $\mho$ can be explicitly expressed as \cite{ABJH6625}
\begin{equation*}
	\begin{split}
		\mho = &\left\{  \mathbf{a} + \mathbf{B}\bx, \mathbf{a}\in\mathbb{R}^3, \mathbf{B}\in M^A \right\}, \\
	\end{split}
\end{equation*}
where $M^A$ is the space of antisymmetric matrices with the size $3\times 3$. A direct calculation shows that the dimension of the space $\mho$ is $12$ and the space $\mho$ is spanned by 
\begin{equation}\label{def-xi}
	\bxi_i = \left\{
	\begin{array}{ll}
		\varkappa_i,  &  \mbox{on} \quad \partial D_1,  \medskip\\
		0     & \mbox{on} \quad \partial D_2,  \medskip
	\end{array}
	\right. \qquad 
	\bxi_{i+6} = \left\{
	\begin{array}{ll}
		0   &  \mbox{on} \quad \partial D_1,  \medskip\\
		\varkappa_i,     & \mbox{on} \quad \partial D_2,  \medskip
	\end{array}
	\right. \qquad
	1\leq i \leq 6,
\end{equation}
where the functions $\varkappa_i$, $1\leq i \leq 6$, are given by 
\begin{equation}\label{eq:dpsi}
	\begin{split}
		\varkappa_1 =	\left[\begin{array}{l}
			1 \\
			0\\
			0
		\end{array}\right],\;\; 
		\varkappa_2 &= \left[\begin{array}{l}
			0 \\
			1 \\
			0
		\end{array}\right],\;\; 
		\varkappa_3 = \left[\begin{array}{c}
			0 \\
			0 \\
			1
		\end{array}\right],  \\
		\varkappa_4 = \left[\begin{array}{c}
			x_2 \\
			-x_1 \\
			0
		\end{array}\right],\;\; 
		\varkappa_5 &= \left[\begin{array}{c}
			x_3 \\
			0 \\
			-x_1
		\end{array}\right],\;\;
		\varkappa_6 = \left[\begin{array}{c}
			0 \\
			x_3 \\
			-x_2
		\end{array}\right].
	\end{split}
\end{equation}

\begin{lem}\label{lem:kek}
	The kernel of the operator $-\mathbf{I}/2 +  \bK_{D}$ coincides with the space $\mho$, where $\bK_{D}$ is the adjoint operator of $ \bK^*_{D}$ (c.f. \cite{ABJH6625}).
\end{lem}
 
\begin{lem}\label{lem:kekske}
	Let $\bzeta_i$ satisfy $\bS_{D}[\bzeta_i] = \bxi_i$ on $\partial D$ with $1\leq i\leq 12$. The kernel of the operator $-\mathbf{I}/2 +  \bK^*_{D}$ is consist of $\bzeta_i$ (c.f. \cite{DKV7353}).
\end{lem}

\subsection{Auxiliary estimates}

In the subsequent analysis, we need the quantitative estimates of $\int_{\partial D} \bzeta_i\bxi_j$, $1\leq i, j\leq 12$. To obtain these estimates, we apply Lemma \ref{lem:kekske}, which yields that ${\bf w}_i:=\bS_{D}[\bzeta_i]$ is the unique solution of the Dirichlet scattering system as follows (c.f. \cite{ABJH6625}): 
\begin{align}\label{eq-wi}
\begin{cases}
\mathcal{L}_{ {\lambda}, {\mu}}{\bf w}_i=0,  &  \bx\in D^{e},\\
{\bf w}_i=\bxi_i,    & \bx\in\partial D,\\
{\bf w}_i(\bx)=\Ocal\left(|\bx|^{-1}\right),& \mbox{as}~|\bx|\rightarrow\infty.
\end{cases}
\end{align}
Then it follows from the jump condition \eqref{eq:jump} and Lemma \ref{lem:kekske} that
\begin{align}\label{int-zeta}
\int_{\partial D} \bzeta_i\bxi_j=\int_{\partial D}\left(\frac{1}{2}\bI +  \bK_{D}^*\right)[\bzeta_i]\bxi_j
=\int_{\partial D}\frac{\partial\bS_{ D}[\bzeta_i]}{\partial{\bnu}}\Bigg|_{+}\bxi_j=
\int_{\partial D}\frac{\partial\bw_i}{\partial{\bnu}}\Bigg|_{+}\bxi_j.
\end{align}
Moreover, by using the definitions of $\mathcal{L}_{ {\lambda}, {\mu}}$ and  $\bw_i$ in \eqref{op:lame} and \eqref{eq-wi}, respectively, and the integration by parts, we obtain
\begin{equation}\label{def-Eij}
\int_{\partial D}\frac{\partial\bw_i}{\partial{\bnu}}\big|_{+}\bxi_j=-\int_{D^e}\left(\mathbb{C}e(\bw_i),e(\bw_j)\right)=:-\mathcal{E}_{i,j}\quad i,j=1,\dots,12.
\end{equation}


To analyze the estimates of $\mathcal{E}_{i,j}$ in \eqref{def-Eij}, we provide more characteristics of the two domains $D_1$ and $D_2$. Let $\varepsilon$ denote the distance between the two domains, namely, 
\begin{equation}\label{eq:disd12}
	\varepsilon:=\mbox{dist}(\partial D_1,\partial D_2).
\end{equation}
We write $(\bx' ,x_3)$ for $\bx = (x_1, x_2, x_3) \in \mathbb{R}^3$
with $\bx' = (x_1, x_2) \in \mathbb{R}^2 $ and $ x_3 \in \mathbb{R}$. 
Denote by $B'_{r}$ a disk on the plane $x_3=0$ with the radius $r$ centered at the origin.
By a translation and rotation of coordinates (if necessary), there exists a constant $R_0$ independent of $\varepsilon$, such that the sections of $\partial D_{1}$ and $\partial D_{2}$ near the origin, respectively, can be represented by
\begin{align}\label{h1h2}
x_{3}=\frac{\varepsilon}{2}+h_{1}(\bx')\quad\mbox{and}\quad x_{3}=-\frac{\varepsilon}{2}+h_{2}(\bx'),\quad\mbox{for}~|\bx'|<2R_0.
\end{align}
In \eqref{h1h2}, the functions $h_i\in C^{2,\alpha}(B'_{2R_0})$, $i=1,2$, have the expressions 
\begin{align}\label{convexity}
h_{i}(\bx')=
(-1)^{i+1}\frac{\kappa}{2}|\bx'|^{2}+O(|\bx'|^{2+\alpha}),
\end{align}
and can be controlled by (c.f. \cite{LX})
\begin{equation*}
\|h_{1}\|_{C^{2,\alpha}(B'_{2R_0})}+\|h_{2}\|_{C^{2,\alpha}(B'_{2R_0})}\leq C,
\end{equation*}
where $C$ is a positive constant independent of $\varepsilon$ and $\kappa$ is the curvature of $\partial D$. Throughout the paper, the constant $C$ is independent of $\varepsilon$, and may vary from line to line in various inequalities. Here we would like to remark that our method can be applied to deal with the
more general inclusions case, say, $h_{i}(\bx')=
(-1)^{i+1}\kappa_i|\bx'|^{2}+O(|\bx'|^{2+\alpha})$ with two positive constants $\kappa_1$ and $\kappa_2$.
For $0<r\leq\,2R_0$, we define the narrow region between $\partial{D}_{1}$ and $\partial{D}_{2}$ as follows:
\begin{equation}\label{narrowreg}
\Omega_r:=\left\{(\bx',x_{3})\in \mathbb{R}^{3}: -\frac{\varepsilon}{2}+h_{2}(\bx')<x_{3}<\frac{\varepsilon}{2}+h_{1}(\bx'),~|\bx'|<r\right\},
\end{equation}
and the vertical distance between $\partial{D}_{1}$ and $\partial{D}_{2}$ is denoted by
\begin{equation}\label{delta_x'}
\delta(\bx'):=\varepsilon+h_{1}(\bx')-h_{2}(\bx')=\varepsilon+\kappa|\bx'|^{2}+O(|\bx'|^{2+\alpha}).
\end{equation}

For the subsequent analysis, we need some assumptions on the structure $D$.
\begin{asum}\label{asu:asd}
We assume the dimer $D$ has the following symmetric assumptions:
\begin{enumerate}
    \item the domain $D$ is symmetric with respect to the axis $x_1$, i.e. $$D(x_1, x_2, x_3) = D(x_1, -x_2, -x_3);$$
    \item the domain $D$ is symmetric with respect to the axis $x_2$, i.e. $$D(x_1, x_2, x_3) = D(-x_1, x_2, -x_3);$$
    \item the domain $D$ is symmetric with respect to the plane $x_3=0$, i.e. $$D(x_1, x_2, x_3) = D(x_1, x_2, -x_3).$$
\end{enumerate}
\end{asum}

\begin{prop}\label{prop-asymp}
Letting $\bw_i$ with $i=1,\dots,12$ be the solution of \eqref{eq-wi}, there holds that
\begin{enumerate}[(i)]
\item for sufficiently small $\varepsilon>0$, 
\begin{equation*}
\mathcal{E}_{i,i}=\frac{\mu\,\pi}{\kappa}|\log\varepsilon|+\mathcal{C}_{i,i}+o(1),~i=1,2,7,8,
\end{equation*}
\begin{equation*}
\mathcal{E}_{i,j}=-\frac{\mu\,\pi}{\kappa}|\log\varepsilon|+\mathcal{C}_{i,j}+o(1),\quad (i,j)=(1,7), (2,8), (7,1), (8,2), 
\end{equation*}
\begin{equation*}
\mathcal{E}_{i,i}=\frac{(\lambda+2\mu)\pi}{\kappa}|\log\varepsilon|+\mathcal{C}_{i,i}+o(1),\quad i=3,9,
\end{equation*}
\begin{equation*}
\mathcal{E}_{i,j}=-\frac{(\lambda+2\mu)\pi}{\kappa}|\log\varepsilon|+\mathcal{C}_{i,j}+o(1),\quad (i,j)=(3,9), (9,3),
\end{equation*}
where $\mathcal{C}_{i,i}$ and $\mathcal{C}_{i,j}$ are constants independent of $\varepsilon$;
\item for $i=1,\dots,6$, 
\begin{equation}\label{sym-Ei6}
\mathcal{E}_{i,i+6}=\mathcal{E}_{i+6,i},
\end{equation}
for $i=1,2$,
\begin{align}\label{sym-Ei}
\begin{split}
\mathcal{E}_{i,i+4}=\mathcal{E}_{i+4,i},\quad\mathcal{E}_{i+6,i+10}=\mathcal{E}_{i+10,i+6},\\
\mathcal{E}_{i,i+10}=\mathcal{E}_{i+10,i},\quad\mathcal{E}_{i+4,i+6}=\mathcal{E}_{i+6,i+4},
\end{split}
\end{align}
and for $i=4,5,6,10,11,12$, 
$$\frac{1}{C}\leq \mathcal{E}_{i,i}\leq C;$$
\item under Assumption \ref{asu:asd}, for $i=1,5,7,11$, $j=2,3,4,6,8,9,10,12$, and $i=2,6,8,12$, $j=1,3,4,5,7,9,10,11$, and $i=3,9$, $j=1,2,4,5,6,7,8,10,11,12$, and $i=4,10$, $j=1,2,3,5,6,7,8,9,11,12$,  
$$\mathcal{E}_{i,j}=0.$$
Moreover, we have 
\begin{align*}
\mathcal{E}_{i,i}=\mathcal{E}_{i+6,i+6},\quad i=1,\dots,6,\\
\mathcal{E}_{i,i+4}=-\mathcal{E}_{i+6,i+10},\quad\mathcal{E}_{i,i+10}=-\mathcal{E}_{i+6,i+4},\quad i=1,2;
\end{align*}
\item under Assumption \ref{asu:asd}, we have for $i=1,\dots,6$,
\begin{equation}\label{sum-Eii}
0\leq\mathcal{E}_{i,i}\pm\mathcal{E}_{i,i+6}\leq C,\quad 0\leq\mathcal{E}_{i+6,i+6}\pm\mathcal{E}_{i+6,i}\leq C.
\end{equation}
\end{enumerate}
\end{prop}

The proof is inspired by the main idea in \cite{BLL2017,LX} and is given in Appendix \ref{Appendix}. 

\begin{rem}
Our arguments in Appendix \ref{Appendix} work well for $C^{2,\alpha}$ resonators with $m$-convex boundaries, that is, the assumption \eqref{convexity} becomes 
\begin{align*}
h_{i}(\bx')=
(-1)^{i+1}\frac{\kappa}{2}|\bx'|^{m}+O(|\bx'|^{m+1}),
\quad i=1,2,~m>2.
\end{align*}
See, for instance, \cite{LX}. We leave the details to the interested reader.
\end{rem}

\section{Resonant frequencies and resonant modes}\label{sec:resonant}

In this section, we study the resonant phenomenon of the system \eqref{eq:or2} including deriving the  resonant frequencies and the corresponding eigenmodes. To that end,
we need first investigate the properties of the matrix $\Bcal$ defined by
\begin{equation}\label{eq:defb}
    \Bcal_{i,j}=\int_D  \bS_{ D} [\bzeta_i]\bxi_j,
\end{equation}
where $\bxi_j$ and $\bzeta_i$ are given in \eqref{def-xi} and Lemma \ref{lem:kekske}, respectively. 
\begin{prop}\label{prop:exob}
    The matrix $\Bcal$ defined in \eqref{eq:defb} has the following properties:
    \begin{enumerate}[(i)]
        \item the matrix $\Bcal$ is symmetric, i.e. $\Bcal_{i,j}=\Bcal_{j,i}$;
        \item some of the elements of the matrix $\Bcal$ satisfy 
        \[
        \Bcal_{i,j} =
\begin{cases}
|D_1|,\quad  &i=j=1,2,3,\\
\int_{D_1} |\varkappa_i|^2, \quad  & i=j=4,5,6, \\
|D_2|,\quad & i=j=7,8,9,\\
\int_{D_2} |\varkappa_{i-6}|^2, \quad & i=j=10, 11, 12 \\
0,\quad & i=1,\dots,6, j=7,\dots,12, \\
0,\quad & i=7,\dots,12, j=1,\dots,6, \\
\int_{D_1} x_3, \quad \mbox{for pairs} \;\; \{i,j\} & \{1,5\}, \{2,6\}, \{5,1\}, \{6,2\},  \\
\int_{D_2} x_3, \quad \mbox{for pairs} \;\; \{i,j\}  & \{7,11\}, \{8,12\}, \{11,7\}, \{12,8\},  \\
0, \quad \mbox{for pairs} \;\; \{i,j\}  & \{1,2\}, \{1,3\}, \{2,3\}, \{7,8\}, \{7,9\}, \{8,9\},  \\
 0, \quad \mbox{for pairs} \;\; \{i,j\}  & \{2,1\}, \{3,1\}, \{3,2\}, \{8,7\}, \{9,7\}, \{9,8\},  \\
\end{cases}
        \]
         where $\varkappa_i$ is given in \eqref{eq:dpsi};
    \item under Assumption \ref{asu:asd}, the other elements of the matrix $\Bcal$ vanish. 
    \end{enumerate}
   
\end{prop}

\begin{proof}
    From the choice of $\bzeta_i$ in Lemma \ref{lem:kekske} and the definition of $\Bcal_{i,j}$ in \eqref{eq:defb}, we have that
    \begin{equation}\label{Bi-j}
    \Bcal_{i,j} = \int_D  \bS_{ D} [\bzeta_i]\bxi_j=\int_D  \bxi_i\bxi_j,
    \end{equation}
    which proves the statement $(i)$. 
    
    The values of the matrix $\Bcal$ in $(ii)$ follow from the direct calculations with the help of the expressions of $\bxi_i$ in \eqref{def-xi} and $\Bcal_{i,j}$ in \eqref{Bi-j}. 

    Next, we prove the third statement. 
    A direct calculation shows that 
    \[
    \Bcal_{1,4} =  \int_{D_1} x_2.
    \]
    From $(1)$ and $(3)$ in Assumption \ref{asu:asd}, one can conclude the domain $D_1$ is symmetric with respect to the plane $x_2=0$. Thus one has that 
    \[
    \int_{D_1} x_2 = \int_{D_1} -x_2,
    \]
   which shows that $\Bcal_{1,4}=0$. The other elements of the matrix $\Bcal$ can be proved following the same discussion. This completes the proof. 
\end{proof}

Now, we are in a position to derive the resonant frequencies and the corresponding eigenvectors for the system \eqref{eq:or2}. To clearly state the results, we first define some functions by 
\begin{equation}\label{eq:bvbr}
\begin{split}
    \bvarphi^{(1)} = \left( \bzeta_3, \bzeta_9 \right)^t, & \qquad \bvarphi^{(3)} = \left( \bzeta_1, \bzeta_{5}, \bzeta_7, \bzeta_{11} \right)^t, \\
    \bvarphi^{(2)} = \left( \bzeta_4, \bzeta_{10} \right)^t, & \qquad \bvarphi^{(4)} = \left( \bzeta_2, \bzeta_{6}, \bzeta_8, \bzeta_{12} \right)^t.
\end{split}
\end{equation}

\begin{thm}\label{thm:renfre}
    Under Assumption \ref{asu:asd} for the structure $D$, consider the system \eqref{eq:or2} with parameters given in \eqref{eq:painma} and \eqref{eq:conpara}. 
    Then the system has $12$ subwavelength resonant frequencies. The first two resonant frequencies are
    \[
    \omega_1 = \sqrt{\frac{2(\lambda + 2\mu)\pi}{\kappa\rho |D_1|}} \sqrt{|\log\varepsilon|\eta} \left(1 + o(1) \right), \quad  \omega_2 = \sqrt{\frac{(\Ccal_{3,3} + \Ccal_{3,9})\eta}{\rho|D_1|}}  \left(1 + o(1) \right).
    \]
    The corresponding eigenvectors are 
    \begin{equation}\label{phi12}
      \bvarphi_1 = \left( \bbe^{(1)}_1 \right)^t \bvarphi^{(1)} \quad \mbox{and} \quad  \bvarphi_2 = \left( \bbe^{(1)}_2 \right)^t \bvarphi^{(1)},
    \end{equation}
    where $\bbe^{(1)}_i$ with $i=1,2$ are given in \eqref{eigenvector1} and \eqref{eigenvector2}, respectively. 
    The next two resonant frequencies are given by 
    \[
    \omega_3 = \sqrt{\frac{(\Ecal_{4,4} - \Ecal_{4,10})\eta }{\rho \Bcal_{4,4}}}\left(1 + o(1) \right), \quad  \omega_4 = \sqrt{\frac{(\Ecal_{4,4} + \Ecal_{4,10})\eta }{\rho \Bcal_{4,4}}} \left(1 + o(1) \right).
    \]
    The corresponding eigenvectors are 
    \begin{equation}\label{phi34}
      \bvarphi_3 = \left( \bbe^{(2)}_1 \right)^t \bvarphi^{(2)} \quad \mbox{and} \quad  \bvarphi_4 = \left( \bbe^{(2)}_2 \right)^t \bvarphi^{(2)},
    \end{equation}
    where $\bbe^{(2)}_i$ with $i=1,2$ are given in \eqref{eigenvector21} and \eqref{eigenvector22}, respectively. 
    The next four resonant frequencies are given by 
    \[
    \omega_5 =  \sqrt{\frac{(\Ecal_{5,5} + \Ecal_{5,11})\eta }{\rho \Bcal_{5,5}}} \left(1 + o(1) \right), \quad  \omega_6 = \sqrt{\frac{2\mu\pi}{\kappa\rho}} \sqrt{   \frac{\Bcal_{5,5} }{  \Bcal_{1,1}\Bcal_{5,5} -\Bcal_{1,5}^2 }} \sqrt{|\log\varepsilon|\eta} \left(1 + o(1) \right),
    \]
    \[
    \omega_7 = \sqrt{\frac{\eta}{\rho} d_1} (1+o(1)), \qquad\qquad  \omega_8 = \sqrt{\frac{\eta}{\rho} d_2} (1+o(1)), 
    \]
    where $d_1$ and $d_2$ are of order $\Ocal(1)$, and defined in \eqref{eq:defd1} and \eqref{eq:defd2}, respectively. 
    The corresponding eigenvectors are 
    \begin{equation}\label{phi5678}
      \bvarphi_5 = \left( \bbe^{(3)}_1 \right)^t \bvarphi^{(3)}, \;\;  \bvarphi_6 = \left( \bbe^{(3)}_2 \right)^t \bvarphi^{(3)}, \;\;  \bvarphi_7 = \left( \bbe^{(3)}_3 \right)^t \bvarphi^{(3)}, \;\;  \bvarphi_8 = \left( \bbe^{(3)}_4 \right)^t \bvarphi^{(3)}, \; 
    \end{equation}
    where $\bbe^{(3)}_i$ with $i=1,2,3,4$ are given in \eqref{eigenvector31}, \eqref{eq:bbe32}, \eqref{eq:bbe33} and \eqref{eq:bbe34}, respectively. 
    The last four resonant frequencies are given by 
    \[
    \omega_9 =  \sqrt{\frac{(\Ecal_{6,6} + \Ecal_{6,12})\eta }{\rho \Bcal_{6,6}}} \left(1 + o(1) \right), \quad  \omega_{10} = \sqrt{\frac{2\mu\pi}{\kappa\rho}} \sqrt{   \frac{\Bcal_{6,6} }{  \Bcal_{2,2}\Bcal_{6,6} -\Bcal_{2,6}^2 }} \sqrt{|\log\varepsilon|\eta} \left(1 + o(1) \right),
    \]
    \[
    \omega_{11} = \sqrt{\frac{\eta}{\rho} d_3} (1+o(1)), \qquad\qquad  \omega_{12} = \sqrt{\frac{\eta}{\rho} d_4} (1+o(1)), 
    \]
    where $d_3$ and $d_4$ are of order $\Ocal(1)$, and defined in \eqref{eq:defd3} and \eqref{eq:defd4}, respectively. 
    The corresponding eigenvectors are 
    \begin{equation}\label{phi10}
      \bvarphi_9 = \left( \bbe^{(4)}_1 \right)^t \bvarphi^{(4)}, \;\;  \bvarphi_{10} = \left( \bbe^{(4)}_2 \right)^t \bvarphi^{(4)}, \;\;  \bvarphi_{11} = \left( \bbe^{(4)}_3 \right)^t \bvarphi^{(4)}, \;\;  \bvarphi_{12} = \left( \bbe^{(4)}_4 \right)^t \bvarphi^{(4)}, \; 
    \end{equation}
    where $\bbe^{(4)}_i$ with $i=1,2,3,4$ are given in \eqref{bbe114}, \eqref{bbe224}, \eqref{eq:bbe43} and \eqref{eq:bbe44}, respectively. 

    Here, the functions $\bvarphi^{(i)}$ with $1\leq i \leq 4$ are defined in \eqref{eq:bvbr}.  
    With the index $1\leq i,j\leq 12$, the parameters $\Bcal_{i,j}$ are given in \eqref{eq:defb}, and the parameters $\Ecal_{i,j}$ and $\Ccal_{i,j}$ are given in Proposition \ref{prop-asymp}. The eigenvectors are normalized of the order $\Ocal(1)$ on the boundary $\partial D$, i.e. $\norm{\bvarphi_i}_{L^2(\partial D)}=\Ocal(1)$ with $1\leq i\leq 12$.  
\end{thm}

\begin{proof}
Consider the system \eqref{eq:conre1}, i.e. 
\[
\Acal(\omega,\delta)[\Phi](\bx)=0.
\]
From asymptotic expansions of  operators $\bS_{ D}^{\omega} $ in \eqref{eq:sise} and $\bK_{ D}^{\omega, *}$ in \eqref{eq:npse} with $\omega\ll 1$ and $\delta\ll 1$, the above equation can be written as
\begin{align}\label{eq:exa1}
	\bS_{ D}[\bvarphi - \bpsi] + \tau\omega \bS_{D,1}[\bvarphi] - \omega \bS_{D,1}[\bpsi]   &= \Ocal\left(\omega^2 \right), \\ 
	\label{eq:exa2}
    \left( -\frac{\bI}{2} +  {\bK}_{ D}^{ *} + (\tau\omega)^2 {\bK}_{ D,2}^{ *} \right)[\bvarphi] - \delta \left( \frac{\bI}{2} +  {\bK}_{ D}^{ *}  \right)[\bpsi]  &=\Ocal\left((\tau\omega)^3 + \delta\omega^2\right). 
\end{align}
Comparing the order on both sides of the equation \eqref{eq:exa2}, one has that the leading order of the density function $\bvarphi$  belongs to the kernel of the operator $-\bI/2 +  {\bK}_{ D}^{ *}$, i.e. 
\begin{equation}\label{eq:probva1}
    \bvarphi - \Ocal((\tau\omega)^2) \in \ker (-\bI/2 +  {\bK}_{ D}^{ *} ).
\end{equation}
With the help of the equation \eqref{eq:exa1} and Lemma \ref{lem:insin}, the following holds
\begin{equation}\label{eq:revp1}
	\bvarphi=\bpsi + \Ocal(\omega).
\end{equation}
Substituting the equation \eqref{eq:revp1} into the equation \eqref{eq:exa2} yields that 
\begin{equation}\label{eq:expan1}
	 \left( -\frac{\bI}{2} +  {\bK}_{ D}^{ *}  \right)[\bvarphi] + (\tau\omega)^2 {\bK}_{ D,2}^{ *} [\bpsi] - \delta \left( \frac{\bI}{2} +  {\bK}_{ D}^{ *}  \right)[\bpsi] = \Ocal\left((\tau\omega)^2\omega + \delta\omega^2 \right).
\end{equation}
Then multiplying $\bxi_i$ on both sides of the equation \eqref{eq:expan1} and integrating on the boundary $\partial D$ give that 
\begin{equation}\label{eq:deref1}
	 -(\tau\omega)^2 \rho \int_D  \bS_{ D} [\bpsi]\bxi_j - \delta\int_{\partial D} \bpsi\bxi_j = \Ocal\left((\tau\omega)^2\omega + \delta\omega^2 \right).
\end{equation}
In the derivation of the last equation, we have used the following identities
\[
\int_{\partial D}  \left( -\frac{\bI}{2} +  {\bK}_{ D}^{ *}  \right)[\bvarphi] \bxi_j =0,
\]
and 
\[
\int_{\partial D} {\bK}_{ D,2}^{ *} [\bpsi] \bxi_j= -\rho \int_D  \bS_{ D} [\bpsi]\bxi_j. 
\]
Thus, to study the resonant frequencies and the corresponding eigenvectors, we just need to solve the equation \eqref{eq:deref1}. From the expressions of the equation \eqref{eq:deref1}, we should first study the properties of the terms $\int_{\partial D} \bpsi\bxi_j$ and $\int_D  \bS_{ D} [\bpsi]\bxi_j$. 
From Lemma \ref{lem:kekske} and the relationships \eqref{eq:probva1} as well as \eqref{eq:revp1}, the density function $\bpsi$ can be written as
\begin{equation}\label{eq:expsi}
    \bpsi=\sum_{i=1}^{12} a_i \bzeta_i + \Ocal(\omega),
\end{equation}
where the coefficients $a_i$ need be determined. To that end, we next study the quantitative behavior of the matrices composed of entries $\int_D  \bS_{ D} [\bzeta_i]\bxi_j$ and $\int_{\partial D} \bzeta_i\bxi_j$.
Indeed, the terms $\int_D  \bS_{ D} [\bzeta_i]\bxi_j$ are the elements of the matrix $\Bcal$ defined in \eqref{eq:defb}. Their corresponding expressions are given in Proposition \ref{prop:exob}. It follows from \eqref{int-zeta} and \eqref{def-Eij} that the corresponding properties of $\int_{\partial D} \bzeta_i\bxi_j$ are given in Proposition \ref{prop-asymp}.

Now we are in a position to analyze the resonant frequencies of the system \eqref{eq:conre1}. From the equation \eqref{eq:deref1}, the resonant frequency $\omega$ satisfies the following equation up to an error of the order $\Ocal\left((\tau\omega)^2\omega + \delta\omega^2 \right)$:
\begin{equation}\label{eq:resys1}
    \left( \Ecal - \frac{(\tau\omega)^2 \rho}{\delta} \Bcal \right)\ba =0.
\end{equation}
In the last equation, $\ba=\{a_i\}_{i=1}^{12}$ is a vector with $a_i$ given in \eqref{eq:expsi}.
From Propositions \ref{prop:exob} and \ref{prop-asymp}, the system \eqref{eq:resys1} can be decomposed into the following four subsystems, 
\begin{equation}\label{eq:subsys1}
     \Jcal^{(1)} \ba^{(1)}=0, \quad  \Jcal^{(2)} \ba^{(2)}=0, \quad  \Jcal^{(3)} \ba^{(3)}=0, \quad  \Jcal^{(4)} \ba^{(4)}=0, \quad 
\end{equation}
where $ \Jcal^{(1)}$ and $ \Jcal^{(2)}$ are two $2\times 2$ matrices, and $ \Jcal^{(3)}$ and $ \Jcal^{(4)}$ are two $4\times 4$ matrices. 
Next, we consider the subsystem in \eqref{eq:subsys1} one by one. 

By Assumption \ref{asu:asd}, the matrix $ \Jcal^{(1)}$ has the following expression
\[
  \Jcal^{(1)}_{1,1} =  \Jcal^{(1)}_{2,2} =\Ecal_{3,3} - \frac{(\tau\omega)^2 \rho}{\delta} \Bcal_{3,3}, \qquad  \Jcal^{(1)}_{1,2} =  \Jcal^{(1)}_{2,1} =\Ecal_{3,9}, 
\]
and $\ba^{(1)}$ has the expression $\ba^{(1)} = (a_3, a_9)^t$. Hence, the first eigenvalue is given by 
\[
\gamma^{(1)}_1 =  \Jcal^{(1)}_{1,1} -  \Jcal^{(1)}_{1,2}, 
\]
which means that the first resonant frequency is
\[
 \omega^{(1)}_1 = \sqrt{\frac{2(\lambda + 2\mu)\pi}{\kappa\rho |D_1|}} \sqrt{|\log\varepsilon|\eta} \left(1 + o(1) \right).
\]
The corresponding eigenvector is $\bvarphi_1 = ( \bbe^{(1)}_1 )^t \bvarphi^{(1)}$ with
\begin{equation}\label{eigenvector1}
\bbe^{(1)}_1 = (-1, 1)^t, \qquad  \bvarphi^{(1)} = \left( \bzeta_3, \bzeta_9 \right)^t.
\end{equation}
The second eigenvalue is 
\[
\gamma^{(1)}_2 =  \Jcal^{(1)}_{1,1} +  \Jcal^{(1)}_{1,2}, 
\]
which shows that the resonant frequency is
\[
  \omega^{(1)}_2 = \sqrt{\frac{(\Ccal_{3,3} + \Ccal_{3,9})\eta}{\rho|D_1|}}  \left(1 + o(1) \right).
\]
The corresponding eigenvector is $\bvarphi_2 = ( \bbe^{(1)}_2 )^t \bvarphi^{(1)}$ with
\begin{equation}\label{eigenvector2}
 \bbe^{(1)}_2 = (1, 1)^t.
\end{equation}

Next, we deal with the second equation in \eqref{eq:subsys1}.  
The matrix $ \Jcal^{(2)}$ has the following expression
\[
  \Jcal^{(2)}_{1,1} =  \Jcal^{(2)}_{2,2} =\Ecal_{4,4} - \frac{(\tau\omega)^2 \rho}{\delta} \Bcal_{4,4}, \qquad  \Jcal^{(2)}_{1,2} =  \Jcal^{(2)}_{2,1} =\Ecal_{4,10}, 
\]
and $\ba^{(2)}$ is given by $\ba^{(2)} = (a_4, a_{10})^t$. Therefore, the first eigenvalue is 
\[
\gamma^{(2)}_1 =  \Jcal^{(2)}_{1,1} -  \Jcal^{(2)}_{1,2}, 
\]
which gives that the resonant frequency is
\[
  \omega^{(2)}_1 = \sqrt{\frac{(\Ecal_{4,4} - \Ecal_{4,10})\eta }{\rho \Bcal_{4,4}}}\left(1 + o(1) \right).
\]
The corresponding eigenvector is $\bvarphi_3 = ( \bbe^{(2)}_1 )^t \bvarphi^{(2)}$ with
\begin{equation}\label{eigenvector21}
\bvarphi^{(2)} = \left( \bzeta_4, \bzeta_{10} \right)^t, \qquad  \bbe^{(2)}_1 = (-1, 1)^t.
\end{equation}
The second eigenvalue is 
\[
\gamma^{(2)}_2 =  \Jcal^{(2)}_{1,1} +  \Jcal^{(2)}_{1,2}, 
\]
which yields the related resonant frequency 
\[
  \omega^{(2)}_2 = \sqrt{\frac{(\Ecal_{4,4} + \Ecal_{4,10})\eta }{\rho \Bcal_{4,4}}} \left(1 + o(1) \right).
\]
The corresponding eigenvector is $\bvarphi_4 = ( \bbe^{(2)}_2 )^t \bvarphi^{(2)}$ with 
\begin{equation}\label{eigenvector22}
\bbe^{(2)}_2 = (1, 1)^t.
\end{equation}

Then, we deal with the third equation in \eqref{eq:subsys1}.  
The matrix $ \Jcal^{(3)}$ has the following expressions
\[
  \Jcal^{(3)}_{1,1} =  \Jcal^{(3)}_{3,3} =\Ecal_{1,1} - \frac{(\tau\omega)^2 \rho}{\delta} \Bcal_{1,1}, \qquad   \Jcal^{(3)}_{1,3} =  \Jcal^{(3)}_{3,1} =\Ecal_{1,7} , 
\]
\[
  \Jcal^{(3)}_{2,2} =  \Jcal^{(3)}_{4,4} =\Ecal_{5,5} - \frac{(\tau\omega)^2 \rho}{\delta} \Bcal_{5,5}, \qquad   \Jcal^{(3)}_{2,4} =  \Jcal^{(3)}_{4,2} =\Ecal_{5,11} , 
\]
\[
  \Jcal^{(3)}_{1,2} =  \Jcal^{(3)}_{2,1} =- \Jcal^{(3)}_{3,4} = - \Jcal^{(3)}_{4,3} = \Ecal_{1,5} - \frac{(\tau\omega)^2 \rho}{\delta} \Bcal_{1,5},
\]
\[
  \Jcal^{(3)}_{1,4} =  \Jcal^{(3)}_{4,1} =- \Jcal^{(3)}_{2,3} = - \Jcal^{(3)}_{3,2} = \Ecal_{1,11},
\]
and $\ba^{(3)}$ has the expression $\ba^{(3)} = (a_1, a_{5}, a_7, a_{11})^t$.
Analyzing the matrix $ \Jcal^{(3)}$, one has that the first eigenvalue is 
\begin{align*}
\gamma^{(3)}_1 &= \frac{1}{2} \Bigg(     \Jcal^{(3)}_{1,1} -    \Jcal^{(3)}_{1,3} +    \Jcal^{(3)}_{2,2}+    \Jcal^{(3)}_{2,4} \\
&\qquad\qquad -\sqrt{(-    \Jcal^{(3)}_{1,1} +    \Jcal^{(3)}_{1,3}+    \Jcal^{(3)}_{2,2}+    \Jcal^{(3)}_{2,4})^2+4 ( \Jcal^{(3)}_{1,2}+ {   \Jcal^{(3)}_{1,4}})^2} \Bigg), 
\end{align*}
which gives that the related resonant frequency
\[
\omega^{(3)}_1 = \sqrt{\frac{(\Ecal_{5,5} + \Ecal_{5,11})\eta }{\rho \Bcal_{5,5}}} \left(1 + o(1) \right). 
\]
The corresponding eigenvector is $\bvarphi_5 = ( \bbe^{(3)}_1 )^t \bvarphi^{(3)}$ with 
\begin{equation}\label{eigenvector31}
	\bvarphi^{(3)} = \left( \bzeta_1, \bzeta_{5}, \bzeta_7, \bzeta_{11} \right)^t, \quad 
\bbe^{(3)}_1 = \left(\bbe^{(3)}_{1,1}, \; \bbe^{(3)}_{1,2} , \;  \bbe^{(3)}_{1,3} , \; \bbe^{(3)}_{1,4} \right)^t.
\end{equation}
In \eqref{eigenvector31}, the parameters have the expressions
\begin{equation}\label{bbe113}
\bbe^{(3)}_{1,1} = \frac{ (\Ecal_{5,5} + \Ecal_{5,11} ) \Bcal_{1,5} - ( \Ecal_{1,5} + \Ecal_{1,11} ) \Bcal_{5,5}  }{ \Bcal_{5,5}  } \frac{\kappa}{2\mu\pi} \frac{1}{|\log\varepsilon|} (1 + o(1)),
\end{equation}
and 
\[
\bbe^{(3)}_{1,2}= \bbe^{(3)}_{1,4} =1, \quad \bbe^{(3)}_{1,3} = -\bbe^{(3)}_{1,1}.
\]
The second eigenvalue is 
\begin{align*}
\gamma^{(3)}_2 &= \frac{1}{2} \Bigg(     \Jcal^{(3)}_{1,1} -    \Jcal^{(3)}_{1,3} +    \Jcal^{(3)}_{2,2}+    \Jcal^{(3)}_{2,4} \\
&\qquad\qquad+ \sqrt{(-    \Jcal^{(3)}_{1,1} +    \Jcal^{(3)}_{1,3}+    \Jcal^{(3)}_{2,2}+    \Jcal^{(3)}_{2,4})^2+4 (    \Jcal^{(3)}_{1,2}+ {   \Jcal^{(3)}_{1,4}})^2} \Bigg),
\end{align*}
which gives that the resonant frequency is 
\[
\omega^{(3)}_2 = \sqrt{\frac{2\mu\pi}{\kappa\rho}} \sqrt{   \frac{\Bcal_{5,5} }{  \Bcal_{1,1}\Bcal_{5,5} -\Bcal_{1,5}^2 }} \sqrt{|\log\varepsilon|\eta} \left(1 + o(1) \right). 
\]
The corresponding eigenvector is $\bvarphi_6 = ( \bbe^{(3)}_2 )^t \bvarphi^{(3)}$ with
\begin{equation}\label{eq:bbe32}
\bbe^{(3)}_2 = \left(\bbe^{(3)}_{2,1}, \; \bbe^{(3)}_{2,2} , \;  \bbe^{(3)}_{2,3} , \; \bbe^{(3)}_{2,4} \right)^t.
\end{equation}
In \eqref{eq:bbe32}, the parameters share the expressions
\begin{equation}\label{bbe213}
\begin{split}
    \bbe^{(3)}_{2,1} =& - \Bigg( \frac{\sqrt{4 e_1^2+4 e_1 e_2 (\Bcal_{5,5}-\Bcal_{1,1})+e_2^2 \left((\Bcal_{1,1}-\Bcal_{5,5})^2+4 \Bcal_{1,5}^2\right)}+2 e_1 }{2 e_2 \Bcal_{1,5}} \\
     & \qquad \qquad  \times\frac{\Bcal_{5,5}-\Bcal_{1,1}}{2 \Bcal_{1,5}}   \Bigg)  (1 + o(1)) 
\end{split}
\end{equation}
and 
\[
\bbe^{(3)}_{2,3} = -\bbe^{(3)}_{2,1}, \quad \bbe^{(3)}_{2,4}= \bbe^{(3)}_{2,2} =1,
\]
with 
\[
e_1 = \frac{\mu \pi}{\kappa}, \quad  e_2 = \frac{2 e_1\Bcal_{5,5} }{  \Bcal_{1,1}\Bcal_{5,5} -\Bcal_{1,5}^2 }.
\]
From the expression above, it is obvious that $\bbe^{(3)}_{2,3} = -\bbe^{(3)}_{2,1}=\Ocal(1)$.
The third eigenvalue is 
\begin{align*}
\gamma^{(3)}_3 &= \frac{1}{2} \Bigg(     \Jcal^{(3)}_{1,1} +    \Jcal^{(3)}_{1,3} +    \Jcal^{(3)}_{2,2} -    \Jcal^{(3)}_{2,4} \\
&\qquad\qquad- \sqrt{(    \Jcal^{(3)}_{1,1} +    \Jcal^{(3)}_{1,3} -    \Jcal^{(3)}_{2,2}+    \Jcal^{(3)}_{2,4})^2+4 (    \Jcal^{(3)}_{1,2} - {   \Jcal^{(3)}_{1,4}})^2} \Bigg) , 
\end{align*}
which shows that the related resonant frequency is 
\[
\omega^{(3)}_3 = \sqrt{\frac{\eta}{\rho} d_1} (1+o(1)).
\]
In the last equation, $d_1$ enjoys the expression 
\begin{equation}\label{eq:defd1}
    d_1 = \frac{ 2 \Bcal_{1,5} (\Ecal_{1,11}-\Ecal_{1,5})-\Ecal_{5,11} \Bcal_{1,1}+\Bcal_{5,5} (\Ccal_{11}+\Ccal_{17})+\Ecal_{5,5} \Bcal_{1,1}
  + \sqrt{\tilde{d}_1}}{2 \left(\Bcal_{1,1} \Bcal_{5,5} - \Bcal_{1,5}^2 \right)}   
\end{equation}
with 
\begin{equation}\label{eq:deftd1}
    \begin{split}
    \tilde{d}_1 = & (2 \Bcal_{1,5} (\Ecal_{1,11}-\Ecal_{1,5})-\Ecal_{5,11} \Bcal_{1,1}+\Bcal_{5,5} (\Ccal_{11}+\Ccal_{17})+\Ecal_{5,5} \Bcal_{1,1})^2- \\
     &4 \left(\Bcal_{1,5}^2-\Bcal_{1,1} \Bcal_{5,5}\right) \left((\Ecal_{1,11}-\Ecal_{1,5})^2+(\Ecal_{5,11}-\Ecal_{5,5}) (\Ccal_{11}+\Ccal_{17})\right).
\end{split}
\end{equation}
The corresponding eigenvector is $\bvarphi_7 = ( \bbe^{(3)}_3 )^t \bvarphi^{(3)}$ with
\begin{equation}\label{eq:bbe33}
    \bbe^{(3)}_3 = \left(\bbe^{(3)}_{3,1}, \; \bbe^{(3)}_{3,2} , \;  \bbe^{(3)}_{3,3} , \; \bbe^{(3)}_{3,4} \right)^t,
\end{equation}
where the elements have the expressions
\begin{equation}\label{bbe313}
\bbe^{(3)}_{3,1}= \bbe^{(3)}_{3,3} = -\frac{\Ecal_{5,1} - \Ecal_{5,5} + \Bcal_{5,5} d_1}{\Ecal_{1,1} - \Ecal_{1,5} + \Bcal_{1,5} d_1}, \quad \bbe^{(3)}_{3,2}= -\bbe^{(3)}_{3,4}=1,
\end{equation}
with $d_1$ defined in \eqref{eq:defd1}. 
The fourth eigenvalue is 
\begin{align*}
\gamma^{(3)}_4 &= \frac{1}{2} \Bigg(     \Jcal^{(3)}_{1,1} +    \Jcal^{(3)}_{1,3} +    \Jcal^{(3)}_{2,2} -    \Jcal^{(3)}_{2,4} \\
&\qquad\qquad+ \sqrt{(    \Jcal^{(3)}_{1,1} +    \Jcal^{(3)}_{1,3} -    \Jcal^{(3)}_{2,2}+    \Jcal^{(3)}_{2,4})^2+4 (    \Jcal^{(3)}_{1,2} - {   \Jcal^{(3)}_{1,4}})^2} \Bigg), 
\end{align*}
which yields that the resonant frequency is 
\[
\omega^{(3)}_4 = \sqrt{\frac{\eta}{\rho} d_2} (1+o(1)).
\]
In the last formula, the parameter $d_2$ has the expression
\begin{equation}\label{eq:defd2}
    d_2 = \frac{ 2 \Bcal_{1,5} (\Ecal_{1,11}-\Ecal_{1,5})-\Ecal_{5,11} \Bcal_{1,1}+\Bcal_{5,5} (\Ccal_{11}+\Ccal_{17})+\Ecal_{5,5} \Bcal_{1,1}
  - \sqrt{\tilde{d}_1}}{2 \left(\Bcal_{1,1} \Bcal_{5,5} - \Bcal_{1,5}^2 \right)}   
\end{equation}
with ${\tilde{d}_1}$ given in \eqref{eq:deftd1}.
The corresponding eigenvector is 
\begin{equation}\label{eq:bbe34}
    \bbe^{(3)}_4 = \left(\bbe^{(3)}_{4,1}, \; \bbe^{(3)}_{4,2} , \;  \bbe^{(3)}_{4,3} , \; \bbe^{(3)}_{4,4} \right)^t,
\end{equation}
where the parameters are given by 
\begin{equation}\label{bbe413}
\bbe^{(3)}_{4,1}= \bbe^{(3)}_{4,3} = -\frac{\Ecal_{5,1} - \Ecal_{5,5} + \Bcal_{5,5} d_2}{\Ecal_{1,1} - \Ecal_{1,5} + \Bcal_{1,5} d_2}, \quad \bbe^{(3)}_{4,2}= -\bbe^{(3)}_{4,4}=1,
\end{equation}
with $d_2$ defined in \eqref{eq:defd2}.

Now, let us deal with the fourth equation in \eqref{eq:subsys1}.  
The matrix $ \Jcal^{(4)}$ has the following expressions
\[
  \Jcal^{(4)}_{1,1} =  \Jcal^{(4)}_{3,3} =\Ecal_{2,2} - \frac{(\tau\omega)^2 \rho}{\delta} \Bcal_{2,2}, \qquad   \Jcal^{(4)}_{1,3} =  \Jcal^{(4)}_{3,1} =\Ecal_{2,8} , 
\]
\[
  \Jcal^{(4)}_{2,2} =  \Jcal^{(4)}_{4,4} =\Ecal_{6,6} - \frac{(\tau\omega)^2 \rho}{\delta} \Bcal_{6,6}, \qquad   \Jcal^{(4)}_{2,4} =  \Jcal^{(4)}_{4,2} =\Ecal_{6,12} , 
\]
\[
  \Jcal^{(4)}_{1,2} =  \Jcal^{(4)}_{2,1} =- \Jcal^{(4)}_{3,4} = - \Jcal^{(4)}_{4,3} = \Ecal_{2,6} - \frac{(\tau\omega)^2 \rho}{\delta} \Bcal_{2,6}
\]
\[
  \Jcal^{(4)}_{1,4} =  \Jcal^{(4)}_{4,1} =- \Jcal^{(4)}_{2,3} = - \Jcal^{(4)}_{3,2} = \Ecal_{2,12},
\]
and $\ba^{(4)}$ is given by $\ba^{(4)} = (a_2, a_{6}, a_8, a_{12})^t$.
Through analyzing the matrix, the first eigenvalue is 
\begin{align*}
\gamma^{(4)}_1 &= \frac{1}{2} \Bigg(     \Jcal^{(4)}_{1,1} -    \Jcal^{(4)}_{1,3} +    \Jcal^{(4)}_{2,2}+    \Jcal^{(4)}_{2,4}\\
&\qquad\qquad-\sqrt{(-    \Jcal^{(4)}_{1,1} +    \Jcal^{(4)}_{1,3}+    \Jcal^{(4)}_{2,2}+    \Jcal^{(4)}_{2,4})^2+4 (    \Jcal^{(4)}_{1,2}+ {   \Jcal^{(4)}_{1,4}})^2} \Bigg), 
\end{align*}
which gives that the associated resonant frequency is 
\[
\omega^{(4)}_1 = \sqrt{\frac{(\Ecal_{6,6} + \Ecal_{6,12})\eta }{\rho \Bcal_{6,6}}} \left(1 + o(1) \right). 
\]
The corresponding eigenvector is $\bvarphi_9 = ( \bbe^{(4)}_1 )^t \bvarphi^{(4)}$ with 
\begin{equation}\label{bbe114}
   \bvarphi^{(4)} = \left( \bzeta_2, \bzeta_{6}, \bzeta_8, \bzeta_{12} \right)^t, \quad \bbe^{(4)}_1 = \left(\bbe^{(4)}_{1,1}, \; \bbe^{(4)}_{1,2} , \;  \bbe^{(4)}_{1,3} , \; \bbe^{(4)}_{1,4} \right)^t.
\end{equation}
In the last formula, the parameters have the expressions
\begin{equation}\label{beta114}
\bbe^{(4)}_{1,1} = \frac{ (\Ecal_{6,6} + \Ecal_{6,12} ) \Bcal_{2,12} - ( \Ecal_{2,6} + \Ecal_{2,12} ) \Bcal_{6,6}  }{ \Bcal_{6,6}  } \frac{\kappa}{2\mu\pi} \frac{1}{|\log\varepsilon|} (1 + o(1)),
\end{equation}
and 
\[
\bbe^{(4)}_{1,2}= \bbe^{(4)}_{1,4} =1, \quad \bbe^{(4)}_{1,3} = -\bbe^{(4)}_{1,1}.
\]
The second eigenvalue is 
\begin{align*}
\gamma^{(4)}_2 &= \frac{1}{2} \Bigg(     \Jcal^{(4)}_{1,1} -    \Jcal^{(4)}_{1,3} +    \Jcal^{(4)}_{2,2}+    \Jcal^{(4)}_{2,4} \\
&\qquad\qquad+ \sqrt{(-    \Jcal^{(4)}_{1,1} +    \Jcal^{(4)}_{1,3}+    \Jcal^{(4)}_{2,2}+    \Jcal^{(4)}_{2,4})^2+4 (    \Jcal^{(4)}_{1,2}+ {   \Jcal^{(4)}_{1,4}})^2} \Bigg),
\end{align*}
which gives the related resonant frequency 
\[
\omega^{(4)}_2 = \sqrt{\frac{2\mu\pi}{\kappa\rho}} \sqrt{   \frac{\Bcal_{6,6} }{  \Bcal_{2,2}\Bcal_{6,6} -\Bcal_{2,6}^2 }} \sqrt{|\log\varepsilon|\eta} \left(1 + o(1) \right). 
\]
The corresponding eigenvector is $\bvarphi_{10} = ( \bbe^{(4)}_2 )^t \bvarphi^{(4)}$ with
\begin{equation}\label{bbe224}
    \bbe^{(4)}_2 = \left(\bbe^{(4)}_{2,1}, \; \bbe^{(4)}_{2,2} , \;  \bbe^{(4)}_{2,3} , \; \bbe^{(4)}_{2,4} \right)^t.
\end{equation}
In the last expression, the parameters enjoys the expressions 
\begin{equation}\label{beta214}
\begin{split}
    \bbe^{(4)}_{2,1} =& - \Bigg( \frac{\sqrt{4 e_1^2+4 e_1 e_2 (\Bcal_{6,6}-\Bcal_{2,2})+e_2^2 \left((\Bcal_{2,2}-\Bcal_{6,6})^2+4 \Bcal_{2,6}^2\right)}+2 e_1 }{2 e_2 \Bcal_{2,6}} \\
     & \qquad \qquad \times\frac{\Bcal_{6,6}-\Bcal_{2,2}}{2 \Bcal_{2,6}}   \Bigg)  (1 + o(1)) 
\end{split}
\end{equation}
and 
\[
\bbe^{(4)}_{2,3} = -\bbe^{(4)}_{2,1}, \quad \bbe^{(4)}_{2,4}= \bbe^{(4)}_{2,2} =1,
\]
with 
\[
e_1 = \frac{\mu \pi}{\kappa}, \quad  e_2 = \frac{2 e_1\Bcal_{6,6} }{  \Bcal_{2,2}\Bcal_{6,6} -\Bcal_{2,6}^2 }.
\]
The third eigenvalue is 
\begin{align*}
\gamma^{(4)}_3 &= \frac{1}{2} \Bigg(     \Jcal^{(4)}_{1,1} +    \Jcal^{(4)}_{1,3} +    \Jcal^{(4)}_{2,2} -    \Jcal^{(4)}_{2,4} \\
&\qquad\qquad- \sqrt{(    \Jcal^{(4)}_{1,1} +    \Jcal^{(4)}_{1,3} -    \Jcal^{(4)}_{2,2}+    \Jcal^{(4)}_{2,4})^2+4 (    \Jcal^{(4)}_{1,2} - {   \Jcal^{(4)}_{1,4}})^2} \Bigg) , 
\end{align*}
which shows the associated resonant frequency 
\[
\omega^{(4)}_3 = \sqrt{\frac{\eta}{\rho} d_3} \sqrt{ +\Ocal(1) }.
\]
In the last formula, the parameter $d_3$ has the expression 
\begin{equation}\label{eq:defd3}
    d_3 = \frac{ 2 \Bcal_{2,6} (\Ecal_{2,12}-\Ecal_{2,6})-\Ecal_{6,12} \Bcal_{2,2}+\Bcal_{6,6} (\Ccal_{2,2}+\Ccal_{2,8})+\Ecal_{6,6} \Bcal_{2,2}
  + \sqrt{\tilde{d}_3}}{2 \left(\Bcal_{2,2} \Bcal_{6,6} - \Bcal_{2,6}^2 \right)},   
\end{equation}
with 
\begin{equation}\label{eq:deftd3}
    \begin{split}
    \tilde{d}_3 = & (2 \Bcal_{2,6} (\Ecal_{2,12}-\Ecal_{2,6})-\Ecal_{6,12} \Bcal_{2,2}+\Bcal_{6,6} (\Ccal_{2,2}+\Ccal_{2,8})+\Ecal_{6,6} \Bcal_{2,2})^2- \\
     &4 \left(\Bcal_{2,6}^2-\Bcal_{2,2} \Bcal_{6,6}\right) \left((\Ecal_{2,12}-\Ecal_{2,6})^2+(\Ecal_{6,12}-\Ecal_{6,6}) (\Ccal_{2,2}+\Ccal_{2,8})\right).
\end{split}
\end{equation}
The corresponding eigenvector is $\bvarphi_{11} = ( \bbe^{(4)}_3 )^t \bvarphi^{(4)}$ with
\begin{equation}\label{eq:bbe43}
    \bbe^{(4)}_3 = \left(\bbe^{(4)}_{3,1}, \; \bbe^{(4)}_{3,2} , \;  \bbe^{(4)}_{3,3} , \; \bbe^{(4)}_{3,4} \right)^t.
\end{equation}
In \eqref{eq:bbe43}, the parameters are given by 
\begin{equation}\label{beta314}
\bbe^{(4)}_{3,1}= \bbe^{(4)}_{3,3} = -\frac{\Ecal_{6,2} - \Ecal_{6,6} + \Bcal_{6,6} d_3}{\Ecal_{2,2} - \Ecal_{2,6} + \Bcal_{2,6} d_3}, \quad \bbe^{(4)}_{3,2}= -\bbe^{(4)}_{3,4}=1,
\end{equation}
with $d_3$ defined in \eqref{eq:defd3}. 
The fourth eigenvalue is 
\begin{align*}
\gamma^{(4)}_4 &= \frac{1}{2} \Bigg(     \Jcal^{(4)}_{1,1} +    \Jcal^{(4)}_{1,3} +    \Jcal^{(4)}_{2,2} -    \Jcal^{(4)}_{2,4} \\
&\qquad\qquad+ \sqrt{(    \Jcal^{(4)}_{1,1} +    \Jcal^{(4)}_{1,3} -    \Jcal^{(4)}_{2,2}+    \Jcal^{(4)}_{2,4})^2+4 (    \Jcal^{(4)}_{1,2} - {   \Jcal^{(4)}_{1,4}})^2} \Bigg), 
\end{align*}
which gives that the associated resonant frequency is 
\[
\omega^{(4)}_4 = \sqrt{\frac{\eta}{\rho} d_4} (1+o(1)),
\]
where the parameter $d_4$ satisfies
\begin{equation}\label{eq:defd4}
    d_4 = \frac{ 2 \Bcal_{2,6} (\Ecal_{2,12}-\Ecal_{2,6})-\Ecal_{6,12} \Bcal_{2,2}+\Bcal_{6,6} (\Ccal_{22}+\Ccal_{28})+\Ecal_{6,6} \Bcal_{2,2}
  - \sqrt{\tilde{d}_3}}{2 \left(\Bcal_{2,2} \Bcal_{6,6} - \Bcal_{2,6}^2 \right)}    
\end{equation}
with ${\tilde{d}_3}$ given in \eqref{eq:deftd3}.
The corresponding eigenvector is $\bvarphi_{12} = ( \bbe^{(4)}_4 )^t \bvarphi^{(4)}$ with 
\begin{equation}\label{eq:bbe44}
    \bbe^{(4)}_4 = \left(\bbe^{(4)}_{4,1}, \; \bbe^{(4)}_{4,2} , \;  \bbe^{(4)}_{4,3} , \; \bbe^{(4)}_{4,4} \right)^t.
\end{equation}
In \eqref{eq:bbe44}, the parameters enjoy the expressions 
\begin{equation}\label{beta414}
\bbe^{(4)}_{4,1}= \bbe^{(4)}_{4,3} = -\frac{\Ecal_{6,2} - \Ecal_{6,6} + \Bcal_{6,6} d_4}{\Ecal_{2,2} - \Ecal_{2,6} + \Bcal_{2,6} d_4}, \quad \bbe^{(4)}_{4,2}= -\bbe^{(4)}_{4,4}=1,
\end{equation}
with $d_4$ defined in \eqref{eq:defd4}. Finally, we complete the proof. 
\end{proof}

\begin{rem}\label{rem:disren}
According to Theorem \ref{thm:renfre}, the resonant frequencies $\omega_i$ for $i=2,3,4,5,7,8$, $9,11,12$ are of the order $\mathcal{O}(\eta)$. For the remaining resonant frequencies $\omega_i$ with $i=1,6,10$, if the parameter $\varepsilon$ is selected as $\varepsilon\sim e^{-1/\eta^{1-\gamma}}$ with $\gamma\in(0,1)$, then it follows that $\omega_i=\mathcal{O}(\eta^{\gamma/2})$.
\end{rem}

\begin{rem}\label{rmk-res}
From Theorem \ref{thm:renfre}, the eigenvector $\bvarphi_1$ has the expression $\bvarphi_1=-\bzeta_3 + \bzeta_9$. Therefore, the leading order of the corresponding resonant mode is $\bS_D[\bvarphi_1] = -\bxi_3 + \bxi_9$ from Lemma \ref{lem:kekske}. From the expressions of $\bxi_3$ and $\bxi_9$ defined in \eqref{def-xi}, we have that the two resonators oscillate in antiphase with one another. Moreover, it is noted that the fields $\varkappa_i$ for $1 \leq i \leq 6$ in \eqref{eq:dpsi} are classified as dipolar resonant fields (see \cite{LZ2861}). Therefore, we can conclude that the resonant mode $\bS_D^{\tau\omega}[\bvarphi_1]$ belongs to the quadrupolar resonance. 
Similarly, the leading order of the resonant mode induced by $\bvarphi_2$ is  $\bS_D[\bvarphi_2] = \bxi_3 + \bxi_9$. Thus the two resonators oscillate in phase with one another. Hence, the  resonant mode $\bS_D^{\tau\omega}[\bvarphi_2]$ belongs to the dipolar resonance. Following the same reasoning, the eigenvector $\bvarphi_3$ is associated with the quadrupolar resonance, while $\bvarphi_4$ is associated with dipolar resonance. 
Furthermore, the resonant modes linked to eigenvectors $\bvarphi_i$ for $5\leq i\leq 12$ are hybrid modes that encompass both dipolar and quadrupolar resonances.
Therefore, the resonant modes associated with $\bvarphi_2$ and $\bvarphi_4$ could be used to achieve the negative mass density. The resonant modes associated with $\bvarphi_1$ and $\bvarphi_3$ could be used to achieve the negative shear modulus. The resonant modes associated with $\bvarphi_i$ for $5\leq i\leq 12$ could be used to achieve both the negative mass density and the negative shear modulus \cite{NaMLai2011}. 
\end{rem}


\section{Gradient estimates of eigenmodes}\label{sec-eigenmodes}

In this section, we quantitatively study the blow-up behaviors of the eigenmodes gradient in the region between the two adjacent resonators $D_1$ and $D_2$.
By adopting the same notation as that in Section \ref{sec:resonant}, the eigenmodes can be written as 
\begin{equation}\label{eigenmodes}
{\bu}_i({\bx})={\bS}_{D}[\bvarphi_i](\bx)+\Ocal(\omega_i),\quad i=1,\dots,12,
\end{equation}
where $\bvarphi_i$ is given in Theorem \ref{thm:renfre}. By using \eqref{phi12}, \eqref{phi34} and the definition of $\bzeta_i$ in Lemma \ref{lem:kekske}, we obtain
\begin{align}\label{eq:u12}
\bu_1(\bx)=
\begin{cases}
-\varkappa_3+\Ocal(\omega_1),&\quad\bx\in\partial D_1,\\
\varkappa_3+\Ocal(\omega_1),&\quad\bx\in\partial D_2,
\end{cases}\quad
\bu_2(\bx)=
\begin{cases}
\varkappa_3+\Ocal(\omega_2),&\quad\bx\in\partial D_1,\\
\varkappa_3+\Ocal(\omega_2,&\quad\bx\in\partial D_2,
\end{cases}
\end{align}
\begin{align*}
\bu_3(\bx)=
\begin{cases}
-\varkappa_4+\Ocal(\omega_3),&\quad\bx\in\partial D_1,\\
\varkappa_4+\Ocal(\omega_3),&\quad\bx\in\partial D_2,
\end{cases}\quad
\bu_4(\bx)=
\begin{cases}
\varkappa_4+\Ocal(\omega_4),&\quad\bx\in\partial D_1,\\
\varkappa_4+\Ocal(\omega_4),&\quad\bx\in\partial D_2.
\end{cases}
\end{align*}
It is noted that the eigenmodes ${\bu_1}$ and ${\bu_3}$ exhibit opposite directions on the two resonators $D_1$ and $D_2$. Hence, the gradients $\nabla{\bu_1}$ and $\nabla{\bu_3}$ experience blow-up phenomena as $D_1$ and $D_2$ approach each other closely. In contrast, the eigenmodes ${\bu_2}$ and ${\bu_4}$ possess the same phase on the two resonators, respectively, resulting in the absence of the blow-up for the gradients $\nabla{\bu_2}$ and $\nabla{\bu_4}$.

Next, by making use of \eqref{phi5678}, we have for $i=1,2$,
\begin{align}\label{eq:u5}
\bu_{i+4}(\bx)=
\begin{cases}
\bbe^{(3)}_{i,1}\varkappa_1+\varkappa_5+\Ocal(\omega_{i+4}),&\quad\bx\in\partial D_1,\\
-\bbe^{(3)}_{i,1}\varkappa_1+\varkappa_5+\Ocal(\omega_{i+4}),&\quad\bx\in\partial D_2,
\end{cases}
\end{align}
and for $i=3,4$,
\begin{align*}
\bu_{i+4}(\bx)=
\begin{cases}
\bbe^{(3)}_{i,1}\varkappa_1+\varkappa_5+\Ocal(\omega_{i+4}),&\quad\bx\in\partial D_1,\\
\bbe^{(3)}_{i,1}\varkappa_1-\varkappa_5+\Ocal(\omega_{i+4}),&\quad\bx\in\partial D_2,
\end{cases}
\end{align*}
where $\bbe^{(3)}_{i,1}$, $i=1,2,3,4$, are defined in \eqref{bbe113}, \eqref{bbe213}, \eqref{bbe313}, and \eqref{bbe413}, respectively. By \eqref{phi10}, we derive for $i=1,2$,
\begin{align*}
\bu_{i+8}(\bx)=
\begin{cases}
\bbe^{(4)}_{i,1}\varkappa_2+\varkappa_6+\Ocal(\omega_{i+8}),&\quad\bx\in\partial D_1,\\
-\bbe^{(4)}_{i,1}\varkappa_2+\varkappa_6+\Ocal(\omega_{i+8}),&\quad\bx\in\partial D_2,
\end{cases}
\end{align*}
and for $i=3,4$,
\begin{align*}
\bu_{i+8}(\bx)=
\begin{cases}
\bbe^{(4)}_{i,1}\varkappa_2+\varkappa_6+\Ocal(\omega_{i+8}),&\quad\bx\in\partial D_1,\\
\bbe^{(4)}_{i,1}\varkappa_2-\varkappa_6+\Ocal(\omega_{i+8}),&\quad\bx\in\partial D_2,
\end{cases}
\end{align*}
where $\bbe^{(4)}_{i,1}$, $i=1,2,3,4$, are defined in \eqref{beta114}, \eqref{beta214}, \eqref{beta314}, and \eqref{beta414}, respectively. 
From the above expressions, it is observed that some components of the resonant modes ${\bu_i}$ with $5\leq i \leq 12$ oscillate in antiphase for the two resonators. Therefore, the blow-up phenomenon may occur for the gradient of the eigenmodes as the two resonators are brought closer. In the sequel, we shall quantitatively analyze the blow-up rates for the gradient of the eigenmodes as the contrast parameter $\eta>0$ is small enough.

\begin{thm}\label{thm:blowrate}
Consider the system \eqref{eq:or2} with parameters given in \eqref{eq:painma} and \eqref{eq:conpara}, and let the parameter $\varepsilon$ and the eigenmodes $\bu_i$ be defined in \eqref{eq:disd12} and \eqref{eigenmodes}, respectively. Then the following holds true:
\begin{equation*}
|\nabla\bu_i(\bx)|\leq C,\quad i=2,4;
\end{equation*}
if we choose $\varepsilon\sim e^{-1/\eta^{1-\gamma}}$ for some $\gamma\in(0,1)$, then as $\eta\rightarrow 0$, the gradients of the eigenmodes $\bu_i$ have the following behaviors, 
\begin{equation}\label{gradient-est}
|\nabla\bu_i(\bx)|\sim
\begin{cases}
\frac{1}{\delta(\bx')},&\quad i=1,6,10,\\
\frac{\varepsilon+|\bx'|}{\delta(\bx')},&\quad i=3,7,8,11,12,\\
\frac{1}{\delta(\bx')|\log\varepsilon|},&\quad i=5,9,
\end{cases}
\end{equation}
where $\delta(\bx')$ is defined in \eqref{delta_x'}.
\end{thm}

\begin{proof}
	From the expressions of $\bu_1$ and $\bu_2$ in \eqref{eq:u12} and the definitions of $\bw_i$ in \eqref{eq-wi}, we have that 
\begin{equation*}
\bu_1=-{\bf w}_3+{\bf w}_9+\Ocal(\omega_1),\quad \bu_2={\bf w}_3+{\bf w}_9+\Ocal(\omega_2).
\end{equation*}
As $\eta\rightarrow0$, $|\nabla({\bf w}_3+{\bf w}_9)|$ is bounded since ${\bf w}_3+{\bf w}_9$ takes the same value on $\partial D_1$ and $\partial D_2$; see, for instance, \cite[Theorem 1.1]{JLX2019}.
In contrast, the following holds due to Lemma \ref{lem-Dui} and classical elliptic estimates,
\begin{equation*}
|\nabla({\bf w}_3-{\bf w}_9)(\bx)|\sim\frac{1}{\delta(\bx')},\quad \bx\in D^e,
\end{equation*}
where $\delta(\bx')$ is defined in \eqref{delta_x'}. Thus, we can conclude that 
\begin{equation*}
|\nabla\bu_1(\bx)|\sim\frac{1}{\delta(\bx')},\quad\bx\in D^e,\quad \max_{\bx\in D^e}|\nabla\bu_2(\bx)|\sim 1.
\end{equation*}
Following the same discussion, we have 
\begin{equation*}
|\nabla\bu_3(\bx)|\sim\frac{\varepsilon+|\bx'|}{\delta(\bx')},\quad\bx\in D^e,\quad \max_{\bx\in D^e}|\nabla\bu_4(\bx)|\sim 1.
\end{equation*}
From the expressions of $\bu_5$ in \eqref{eq:u5} and $\bw_i$ in \eqref{eq-wi}, the resonant mode $\bu_5$ can be written as 
\begin{equation*}
\bu_5=\bbe^{(3)}_{1,1}({\bf w}_1-{\bf w}_7)+{\bf w}_5+{\bf w}_{11}+\Ocal(\omega_5).
\end{equation*}
Since $|\nabla({\bf w}_5+{\bf w}_{11})|$ is bounded, and from Lemma \ref{lem-Dui}, we have that 
\begin{equation*}
|\nabla({\bf w}_1-{\bf w}_7)|\sim\frac{1}{\delta(\bx')},\quad\bx\in D^e.
\end{equation*}
The last two formulas together with \eqref{bbe113} imply that 
\begin{equation*}
|\nabla\bu_5(\bx)|\sim\frac{1}{\delta(\bx')|\log\varepsilon|},\quad\bx\in D^e.
\end{equation*}
The others can be proved in the same way and thus we omit the details here. This completes the proof. 
\end{proof}

\begin{rem}
It follows from \eqref{gradient-est} that $|\nabla\bu_i(\bx)|$ with $i=1,6,10$ blow up at the rate of $\frac{1}{\varepsilon}$ in the narrow region $\Omega_{R_0}$ defined in \eqref{narrowreg}. Moreover, they take the maximum at $|\bx'|=0$. In Figure \ref{blowup} (a), the values of $|\nabla\bu_i(\bx)|$ in the red part $\{|\bx'|\leq\sqrt{\kappa^{-1}\varepsilon}\}$ are larger than those in the orange part. For $i=3,7,8,11,12$, we find that, if $\bx\in\Omega_{R_0}$ and $|\bx'|\leq\varepsilon$, then $|\nabla\bu_i(\bx)|$ are bounded, see the blank area in Figure \ref{blowup} (b); if $\bx\in\Omega_{R_0}$ and $|\bx'|>\varepsilon$, then they blow up at the rate of $\frac{1}{\sqrt\varepsilon}$ and attain the maximum at $|\bx'|=\sqrt{\kappa^{-1}\varepsilon}$, see Figure \ref{blowup} (b). 
\end{rem}

\begin{figure}[t]
	\centering
 \subfigure[Cases of $i=1,6,10$]{
 \includegraphics[width=6cm,height=4cm]{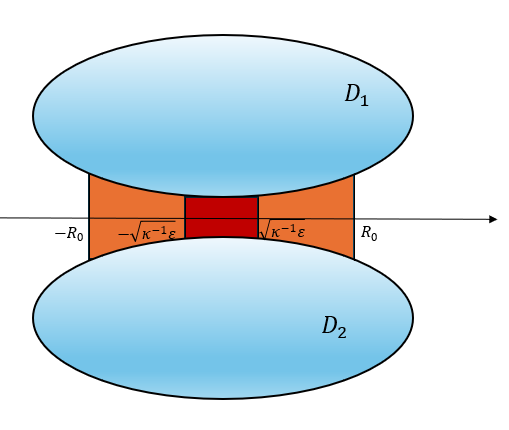}}
 \subfigure[Cases of $i=3,7,8,11,12$]{
 \includegraphics[width=6cm,height=4cm]{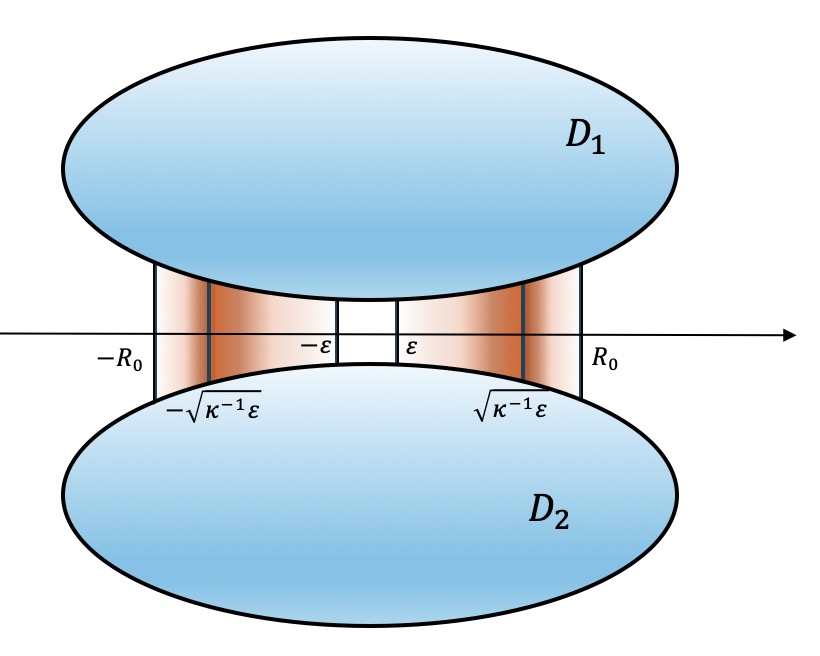}}\caption{Blow-up analysis}
 \label{blowup}
\end{figure}

\begin{rem}\label{rem:debr}
As aforementioned in Remark \ref{rmk-res}, the resonant modes $\bu_i$ with $i=1,3$ and $5\leq i \leq 12$ contain the quadrupolar resonances and can be utilized to realize the negative shear modulus of the metamaterials in elasticity. From gradient estimates of the resonant modes in \eqref{gradient-est}, $|\nabla\bu_i(\bx)|$ with $i=1,6,10$ blow up at the rates of $\frac{1}{\varepsilon}$ and $|\nabla\bu_i(\bx)|$ with $i=5,9$ blow up at the rate of $\frac{1}{\varepsilon\log\varepsilon|}$. Thus, when these modes are used to realize the negative shear modulus, the corresponding metamaterials shall be unstable and can be easily broken when the distance $\varepsilon$ between two hard inclusions is small. In contrast, the gradient estimates $|\nabla\bu_i(\bx)|$ with $i=3,7,8,11,12$ blow up at the rate of  $\frac{1}{\sqrt{\varepsilon \kappa}}$. Thus, if the curvature $\kappa$ of the hard inclusion is chosen at the order of $\Ocal(\frac{1}{\varepsilon})$, the gradient estimates of the resonant modes are bounded and at the order of $\Ocal(1)$. Therefore, applying the resonant modes $\bu_i$ with $i=3,7,8,11,12$ to realize the negative shear modulus results in stable metamaterials, if the curvature of the hard inclusions is properly chosen. 
\end{rem}

\section{Scattering wave fields}

In this section, we study the scattering wave field corresponding to the incident wave field $\bu^i$, and then represent the solution in terms of the subwavelength eigenmodes analyzed in Section \ref{sec-eigenmodes}. 

\begin{thm}\label{thm:scat}

   Consider the system \eqref{eq:or2} with parameters given in \eqref{eq:painma} and \eqref{eq:conpara}.  Let $\bvarphi_i$ with $1\leq i\leq 12$ be the resonant eigenvectors given in Theorem \ref{thm:renfre}. 
   Then under  Assumption \ref{asu:asd},  the total wave field has the expression for $\bx\in\mathbb{R}^3\backslash\overline{D}$,
    \[
\bu(\bx) = \bu^i - {\bS}_{ D}^{\omega} \left[\bS_{ D}^{-1}[\bu^i] \right] + \sum_{i=1}^{12} b_i {\bS}_{ D}[\bvarphi_i] +\Ocal\left((\tau\omega)^3 + \delta\omega \right),
\]
where the coefficients $b_i$ with $1\leq i\leq 12$ are given in \eqref{eq:deb12}, \eqref{eq:deb34}, \eqref{eq:deb58}, and \eqref{eq:deb912}, respectively. 
\end{thm}

\begin{proof}
Consider the system \eqref{eq:or} and assume the pair $(\bvarphi, \bpsi)$ solve the equation. From asymptotic expansions of  operators $\bS_{D}^{\omega} $ in \eqref{eq:sise} and $\bK_{D}^{\omega, *}$ in \eqref{eq:npse}, there holds that 
\begin{align}\label{eq:exas1}
	\bS_{ D}[\bvarphi - \bpsi]   &= \bu^i + \Ocal(\omega), \\ 
	\label{eq:exas2}
    \left( -\frac{\bI}{2} +  {\bK}_{ D}^{ *} + (\tau\omega)^2 {\bK}_{ D,2}^{ *} \right)[\bvarphi] - \delta \left( \frac{\bI}{2} +  {\bK}_{ D}^{ *}  \right)[\bpsi]  &=\Ocal((\tau\omega)^3 + \delta\omega). 
\end{align}
The equation \eqref{eq:exas1} gives that 
\begin{equation}\label{eq:sospsi}
    \bpsi = \bvarphi - \bS_{ D}^{-1}[\bu^i] + \Ocal(\omega).
\end{equation}
Substituting \eqref{eq:sospsi} into the equation \eqref{eq:exas2} yields 
\begin{equation}\label{eq:soexs2}
\begin{split}
    &\left( -\frac{\bI}{2} +  {\bK}_{ D}^{ *} + (\tau\omega)^2 {\bK}_{ D,2}^{ *} \right)[\bvarphi] - \delta \left( \frac{\bI}{2} +  {\bK}_{ D}^{ *}  \right)[\bvarphi] \\
     =& -\delta \left( \frac{\bI}{2} +  {\bK}_{ D}^{ *}  \right)\bS_{ D}^{-1}[\bu^i] + \Ocal\left((\tau\omega)^3 + \delta\omega \right).
\end{split}
\end{equation}
Next, we decompose the density function $\bvarphi$ into
\begin{equation}\label{eq:decss1}
    \bvarphi = \tilde{\bvarphi} + \tilde{\tilde{\bvarphi}} := \sum_{i=1}^{12} b_i \bvarphi_i + \tilde{\tilde{\bvarphi}},
\end{equation}
where $b_i$ will be fixed later. Substituting the decomposition \eqref{eq:decss1} into \eqref{eq:soexs2} and with the help of Lemma \ref{lem:kekske}, we obtain that
\begin{align}\label{eq:soexs3}
     & \left( -\frac{\bI}{2} + {\bK}_{ D}^{ *} \right) [\tilde{\tilde{\bvarphi}}]  +  (\tau\omega)^2 {\bK}_{ D,2}^{ *} [\bvarphi] - \delta \left( \frac{\bI}{2} +  {\bK}_{ D}^{ *}  \right)[\bvarphi] \\
      =& -\delta \left( \frac{\bI}{2} +  {\bK}_{ D}^{ *}  \right)\bS_{ D}^{-1}[\bu^i]\nonumber + \Ocal\left((\tau\omega)^3 + \delta\omega \right).
\end{align}
Hence, from the last equation, we can conclude that 
\[
\norm{\tilde{\tilde{\bvarphi}}}_{L^2(\partial D)} = \Ocal\left(\delta + (\tau\omega)^2\right). 
\]
Multiplying $\bxi_j$ with $1\leq j \leq 12$ on both sides of \eqref{eq:soexs3} and integrating on $\partial D$, we have that 
\begin{equation}\label{eq:soexs4}
     -(\tau\omega)^2 \rho \int_D  \bS_{ D} [\tilde\bvarphi]\bxi_j - \delta\int_{\partial D} \tilde\bvarphi\bxi_j =  -\delta \int_{\partial D}\bS_{ D}^{-1}[\bu^i] \bxi_j  +\Ocal\left((\tau\omega)^3 + \delta\omega \right).
\end{equation}
Moreover, from the proof of Theorem \ref{thm:renfre}, the eigenmodes $\bvarphi_i$ satisfy the following relationship
\begin{equation}\label{eq:resoef1}
    -(\tau\omega_i)^2 \rho \int_D  \bS_{ D} [\bvarphi_i]\bxi_j - \delta\int_{\partial D} \bvarphi_i\bxi_j = \Ocal\left((\tau\omega)^3 + \delta\omega \right).
\end{equation}
Substituting \eqref{eq:resoef1} into \eqref{eq:soexs4} gives that 
\begin{equation}\label{eq:soexs5}
     \rho \tau^2 \sum_{i=1}^{12}(\omega^2 - \omega_i^2) b_i \int_D  \bS_{ D} [\bvarphi_i]\bxi_j  =  \delta \int_{\partial D}\bS_{ D}^{-1}[\bu^i] \bxi_j  +\Ocal\left((\tau\omega)^3 + \delta\omega \right).
\end{equation}
Next, we solve the equation \eqref{eq:soexs5} to obtain the parameters $b_i$ in \eqref{eq:decss1}. From the proof in Theorem \ref{thm:renfre}, the equation \eqref{eq:soexs5} can be decomposed into $4$ sub-equations. 
We first consider the parameters $b_1$ and $b_2$ corresponding to eigenvectors $\bvarphi_1$ and $\bvarphi_2$. Substituting $\bvarphi_1$ and $\bvarphi_2$ into \eqref{eq:soexs5} yields that 
\[
\left(
 (\omega^2 - \omega_1^2) \bbe^{(1)}_1, \; (\omega^2 - \omega_2^2) \bbe^{(1)}_2
\right) \bb^{(1)} = \frac{\delta}{\rho\tau^2} \bg^{(1)},
\]
with $\bb^{(1)} = (b_1, b_2)^t$, and $\bbe^{(1)}_1$ as well as $\bbe^{(1)}_2$ given in Theorem \ref{thm:renfre}. In the last equation, the function $\bg^{(1)}$ has the expression
\[
\bg^{(1)} = \left( \frac{1}{\Bcal_{3,3}} \int_{\partial D}\bS_{ D}^{-1}[\bu^i] \bxi_3, \quad \frac{1}{\Bcal_{9,9}} \int_{\partial D}\bS_{ D}^{-1}[\bu^i] \bxi_9
\right)^t,
\]
with $\Bcal_{3,3}$ and $\Bcal_{9,9}$ given in Proposition \ref{prop:exob}. Solving the above system, we have that
\begin{equation}\label{eq:deb12}
    \begin{split}
    b_1 = \frac{\delta \left(\bbe^{(1)}_1\right)^t\cdot\bg^{(1)} }{2\rho\tau^2 (\omega^2 - \omega_1^2)}, \qquad  b_2 = \frac{\delta \left(\bbe^{(1)}_2\right)^t\cdot\bg^{(1)} }{2\rho\tau^2 (\omega^2 - \omega_2^2)}.
\end{split}
\end{equation}

Following the same discussion, we can obtain that the parameters $b_3$ and $b_4$ are written by 
\begin{equation}\label{eq:deb34}
    \begin{split}
    b_3 = \frac{\delta \left(\bbe^{(2)}_1\right)^t\cdot\bg^{(2)} }{2\rho\tau^2 (\omega^2 - \omega_3^2)}, \qquad  b_4 = \frac{\delta \left(\bbe^{(2)}_2\right)^t\cdot\bg^{(2)} }{2\rho\tau^2 (\omega^2 - \omega_4^2)},
\end{split}
\end{equation}
where $\bbe^{(2)}_1$ and $\bbe^{(2)}_2$ are given in Theorem \ref{thm:renfre}, and the function $\bg^{(2)}$ has the expression
\[
\bg^{(2)} = \left( \frac{1}{\Bcal_{4,4}} \int_{\partial D}\bS_{ D}^{-1}[\bu^i] \bxi_4, \quad \frac{1}{\Bcal_{10,10}} \int_{\partial D}\bS_{ D}^{-1}[\bu^i] \bxi_{10}
\right)^t.
\]

Next, we analyze the coefficients $\bb^{(3)}=(b_5, b_6, b_7, b_8)^t$. Substituting the expressions of $\bvarphi_i$ with $5\leq i \leq 8$ into \eqref{eq:soexs5}, we have that the parameters $\bb^{(3)}$ satisfy the following system
\begin{equation}\label{eq:eqcoesa}
    \Ncal \bb^{(3)} = \delta \bg^{(3)}  +\Ocal\left((\tau\omega)^3 + \delta\omega \right),
\end{equation}
where $\Ncal$ is a $4\times4$ matrix and $\bg^{(3)}$ is given by
\[
\bg^{(3)} = \left( \int_{\partial D}\bS_{ D}^{-1}[\bu^i] \bxi_1, \quad \int_{\partial D}\bS_{ D}^{-1}[\bu^i] \bxi_5, \quad \int_{\partial D}\bS_{ D}^{-1}[\bu^i] \bxi_7, \quad \int_{\partial D}\bS_{ D}^{-1}[\bu^i] \bxi_{11}
\right)^t.
\]
In the equation \eqref{eq:eqcoesa}, the entries of the matrix $\Ncal$ are given as follows
\[
\Ncal_{:,i}=
\left(\begin{array}{c}
			 \bbe^{(3)}_{i,1} \Bcal_{1,1} + \bbe^{(3)}_{i,2} \Bcal_{5,1}  \vspace{0.2cm} \\ 
		   \bbe^{(3)}_{i,1} \Bcal_{1,5} + \bbe^{(3)}_{i,2} \Bcal_{5,5} \vspace{0.2cm} \\
              \bbe^{(3)}_{i,3} \Bcal_{7,7} + \bbe^{(3)}_{i,4} \Bcal_{11,7} \vspace{0.2cm} \\
		   \bbe^{(3)}_{i,3} \Bcal_{7,11} + \bbe^{(3)}_{i,4} \Bcal_{11,11} 
		\end{array}\right) \rho\tau^2 \left(\omega^2 - \omega_{(i+4)}^2 \right), \quad 1\leq i \leq 4,
\]
where $\Ncal_{:,i}$ denotes the $i$-th column of the matrix $\Ncal$. By Assumption \ref{asu:asd}, we finally have that the coefficients can be written as 
\begin{equation}\label{eq:deb58}
    \begin{split}
    b_5 = & \frac{\eta}{\left(\omega^2 - \omega_{5}^2 \right)  } d_5   +\Ocal\left((\tau\omega)^3 + \delta\omega \right), \quad b_6 =  \frac{\eta}{\left(\omega^2 - \omega_{6}^2 \right)  } d_6   +\Ocal\left((\tau\omega)^3 + \delta\omega \right), \\
     b_7 = & \frac{\eta}{\left(\omega^2 - \omega_{5}^2 \right)  } d_7   +\Ocal\left((\tau\omega)^3 + \delta\omega \right), \quad b_8 =  \frac{\eta}{\left(\omega^2 - \omega_{6}^2 \right)  } d_8   +\Ocal\left((\tau\omega)^3 + \delta\omega \right),
\end{split}
\end{equation}
with
\[
d_5 = \frac{\Ncal_{2,2}\left(\bg^{(3)}_1 - \bg^{(3)}_3 \right) - \Ncal_{1,2}\left(\bg^{(3)}_2 + \bg^{(3)}_4 \right) }{2\rho\left(\Ncal_{1,1}\Ncal_{2,2} - \Ncal_{1,2} \Ncal_{2,1} \right)}, \]
\[
d_6 = \frac{\Ncal_{2,1}\left(\bg^{(3)}_3 - \bg^{(3)}_1 \right) + \Ncal_{1,1}\left(\bg^{(3)}_2 + \bg^{(3)}_4 \right) }{2\rho\left(\Ncal_{1,1}\Ncal_{2,2} - \Ncal_{1,2} \Ncal_{2,1} \right)},
\]
\[
d_7 = \frac{\Ncal_{2,4}\left(\bg^{(3)}_1 + \bg^{(3)}_3 \right) - \Ncal_{1,4}\left(\bg^{(3)}_2 - \bg^{(3)}_4 \right)}{2\rho\left(\Ncal_{1,3}\Ncal_{2,4} - \Ncal_{1,4} \Ncal_{2,3} \right)}, \]
\[
d_8 = \frac{-\Ncal_{2,3}\left(\bg^{(3)}_1 + \bg^{(3)}_3 \right) + \Ncal_{1,4}\left(\bg^{(3)}_2 - \bg^{(3)}_4 \right)}{2\rho\left(\Ncal_{1,3}\Ncal_{2,4} - \Ncal_{1,4} \Ncal_{2,3} \right)}.
\]
We can readily check that the parameters $d_i$ with $5\leq i\leq 8$ are of $\Ocal(1)$.

Next, we analyze the coefficients $\bb^{(4)}=(b_9, b_{10}, b_{11}, b_{12})^t$. Substituting the expressions of $\bvarphi_i$ with $9\leq i \leq 12$ into \eqref{eq:soexs5}, there holds that
\begin{equation}\label{eq:eqcoesa1}
    \Qcal \bb^{(4)} = \delta \bg^{(4)}  +\Ocal\left((\tau\omega)^3 + \delta\omega \right),
\end{equation}
where $\Qcal$ is a $4\times4$ matrix and $\bg^{(4)}$ is given by
\[
\bg^{(4)} = \left( \int_{\partial D}\bS_{ D}^{-1}[\bu^i] \bxi_2, \quad \int_{\partial D}\bS_{ D}^{-1}[\bu^i] \bxi_6, \quad \int_{\partial D}\bS_{ D}^{-1}[\bu^i] \bxi_8, \quad \int_{\partial D}\bS_{ D}^{-1}[\bu^i] \bxi_{12}
\right)^t.
\]
In the equation \eqref{eq:eqcoesa1}, the entries of the matrix $\Qcal$ are given by
\[
\Qcal_{:,i}=
\left(\begin{array}{c}
			 \bbe^{(4)}_{i,1} \Bcal_{2,2} + \bbe^{(4)}_{i,2} \Bcal_{6,2}  \vspace{0.2cm} \\ 
		   \bbe^{(4)}_{i,1} \Bcal_{2,6} + \bbe^{(4)}_{i,2} \Bcal_{6,6} \vspace{0.2cm} \\
              \bbe^{(4)}_{i,3} \Bcal_{8,8} + \bbe^{(4)}_{i,4} \Bcal_{12,8} \vspace{0.2cm} \\
		   \bbe^{(4)}_{i,3} \Bcal_{8,12} + \bbe^{(4)}_{i,4} \Bcal_{12,12} 
		\end{array}\right) \rho\tau^2 \left(\omega^2 - \omega_{(i+8)}^2 \right), \quad 1\leq i \leq 4,
\]
where $\Qcal_{:,i}$ denotes the $i$-th column of the matrix $\Qcal$. By Assumption \ref{asu:asd}, we finally have that the coefficients can be written as 
\begin{equation}\label{eq:deb912}
    \begin{split}
    b_9 = & \frac{\eta}{\left(\omega^2 - \omega_{9}^2 \right)  } d_9   +\Ocal\left((\tau\omega)^3 + \delta\omega \right), \quad b_{10} =  \frac{\eta}{\left(\omega^2 - \omega_{10}^2 \right)  } d_{10}   +\Ocal\left((\tau\omega)^3 + \delta\omega \right), \\
     b_{11} = & \frac{\eta}{\left(\omega^2 - \omega_{11}^2 \right)  } d_{11}   +\Ocal\left((\tau\omega)^3 + \delta\omega \right), \quad b_{12} =  \frac{\eta}{\left(\omega^2 - \omega_{12}^2 \right)  } d_{12}   +\Ocal\left((\tau\omega)^3 + \delta\omega \right),
\end{split}
\end{equation}
with 
\[
d_9 = \frac{\Qcal_{2,2}\left(\bg^{(4)}_1 - \bg^{(4)}_3 \right) - \Qcal_{1,2}\left(\bg^{(4)}_2 + \bg^{(4)}_4 \right) }{2\rho\left(\Qcal_{1,1}\Qcal_{2,2} - \Qcal_{1,2} \Qcal_{2,1} \right)}, \]
\[
d_{10} = \frac{\Qcal_{2,1}\left(\bg^{(4)}_3 - \bg^{(4)}_1 \right) + \Qcal_{1,1}\left(\bg^{(4)}_2 + \bg^{(4)}_4 \right) }{2\rho\left(\Qcal_{1,1}\Qcal_{2,2} - \Qcal_{1,2} \Qcal_{2,1} \right)},
\]
\[
d_{11} = \frac{\Qcal_{2,4}\left(\bg^{(4)}_1 + \bg^{(4)}_3 \right) - \Qcal_{1,4}\left(\bg^{(4)}_2 - \bg^{(4)}_4 \right)}{2\rho\left(\Qcal_{1,3}\Qcal_{2,4} - \Qcal_{1,4} \Qcal_{2,3} \right)}, \]
and 
\[
d_{12} = \frac{-\Qcal_{2,3}\left(\bg^{(4)}_1 + \bg^{(4)}_3 \right) + \Qcal_{1,4}\left(\bg^{(4)}_2 - \bg^{(4)}_4 \right)}{2\rho\left(\Qcal_{1,3}\Qcal_{2,4} - \Qcal_{1,4} \Qcal_{2,3} \right)}.
\]
We can directly check that the parameters $d_i$ with $9\leq i\leq 12$ are of $\Ocal(1)$. Finally, we can conclude that the solution to the system \eqref{eq:or2} can be written as 
\[
\bu(\bx) = \bu^i - {\bS}_{ D}^{\omega} \left[\bS_{ D}^{-1}[\bu^i] \right] + \sum_{i=1}^{12} b_i {\bS}_{ D}[\bvarphi_i] +\Ocal\left((\tau\omega)^3 + \delta\omega \right).
\]
Theorem \ref{thm:scat} is proved.
\end{proof}

\appendix

\section{The proof of Proposition \ref{prop-asymp}}\label{Appendix}

This appendix is devoted to the proof of Proposition \ref{prop-asymp}. 
Let us first prove $(iii)$.

\begin{proof}[\bf The proof of $(iii)$ in Proposition \ref{prop-asymp}]
By using the symmetry of the domain with respect to the origin, we take 
$$\bw_{i+6}(x_{1},x_{2},x_{3})=\bw_{i}(-x_{1},-x_{2},-x_{3})\quad\mbox{in}~D^e,\quad i=1,2,3,$$ 
such that 
\begin{align*}
\bw_{i+6}(x_{1},x_{2},x_{3})\big|_{\partial D_{1}}=\bw_{i}(-x_{1},-x_{2},-x_{3})\big|_{\partial D_{1}}&=0,
\end{align*}
\begin{align*}
\bw_{i+6}(x_{1},x_{2},x_{3})\big|_{\partial D_{2}}=\bw_{i}(-x_{1},-x_{2},-x_{3})\big|_{\partial D_{2}}=\bxi_{i},
\end{align*}
and as $|\bx|\rightarrow\infty$,
\begin{align*}
\bw_{i+6}(x_{1},x_{2},x_{3})=\Ocal\left(|\bx|^{-1}\right),\quad\bw_{i}(-x_{1},-x_{2},-x_{3})=\Ocal\left(|\bx|^{-1}\right).
\end{align*}
Combining with the definition of $\mathcal{E}_{i,j}$ in \eqref{def-Eij}, we obtain 
\begin{equation}\label{origin_sym0}
\mathcal{E}_{i,j}=\mathcal{E}_{i+6,j+6},\quad \mathcal{E}_{i,j+6}=\mathcal{E}_{j+6,i},\quad i,j=1,2,3.
\end{equation}
For $j=4,5,6$, we take 
$$\bw_{j+6}(x_{1},x_{2},x_{3})=-\bw_{j}(-x_{1},-x_{2},-x_{3})\quad\mbox{in}~D^e,$$ 
and 
\begin{align*} 
\bw_{j+6}(x_{1},x_{2},x_{3})\big|_{\partial D_{1}}=-\bw_{j}(-x_{1},-x_{2},-x_{3})\big|_{\partial D_{1}}&=0,
\end{align*}
\begin{align*} 
\bw_{j+6}(x_{1},x_{2},x_{3})\big|_{\partial D_{2}}=-\bw_{j}(-x_{1},-x_{2},-x_{3})\big|_{\partial D_{2}}=\bxi_{j}.
\end{align*}
Then we have 
\begin{align}\label{origin_sym1}
\mathcal{E}_{i,j}=-\mathcal{E}_{i+6,j+6},\quad\mathcal{E}_{i,j+6}=-\mathcal{E}_{i+6,j},\quad i=1,2,3,~j=4,5,6,
\end{align}
and 
\begin{align}\label{origin_sym2}
\mathcal{E}_{i,j}=\mathcal{E}_{i+6,j+6},\quad i,j=4,5,6.
\end{align}

In view of the symmetry of the domain with respect to $\{x_{3}=0\}$, we take
\begin{align*}
\bw_{i+6}(x_{1},x_{2},x_{3})=\begin{pmatrix}
\bw_{i+6}^{(1)}(x_{1},x_{2},x_{3})\\\\
\bw_{i+6}^{(2)}(x_{1},x_{2},x_{3})\\\\
\bw_{i+6}^{(3)}(x_{1},x_{2},x_{3})
\end{pmatrix}
=\begin{pmatrix}
\bw_{i}^{(1)}(x_{1},x_{2},-x_{3})\\\\
\bw_{i}^{(2)}(x_{1},x_{2},-x_{3})\\\\
-\bw_{i}^{(3)}(x_{1},x_{2},-x_{3})
\end{pmatrix},\quad i=1,2,4,
\end{align*}
and 
\begin{align*}
\bw_{i+6}(x_{1},x_{2},x_{3})=\begin{pmatrix}
\bw_{i+6}^{(1)}(x_{1},x_{2},x_{3})\\\\
\bw_{i+6}^{(2)}(x_{1},x_{2},x_{3})\\\\
\bw_{i+6}^{(3)}(x_{1},x_{2},x_{3})
\end{pmatrix}
=\begin{pmatrix}
-\bw_{i}^{(1)}(x_{1},x_{2},-x_{3})\\\\
-\bw_{i}^{(2)}(x_{1},x_{2},-x_{3})\\\\
\bw_{i}^{(3)}(x_{1},x_{2},-x_{3})
\end{pmatrix},\quad i=3,5,6.
\end{align*}
Similarly, making use of the symmetry of the domain with respect to the axis $x_1$, we take 
\begin{align*}
\bw_{i+6}(x_{1},x_{2},x_{3})=\begin{pmatrix}
\bw_{i+6}^{(1)}(x_{1},x_{2},x_{3})\\\\
\bw_{i+6}^{(2)}(x_{1},x_{2},x_{3})\\\\
\bw_{i+6}^{(3)}(x_{1},x_{2},x_{3})
\end{pmatrix}
=\begin{pmatrix}
\bw_{i}^{(1)}(x_{1},-x_{2},-x_{3})\\\\
-\bw_{i}^{(2)}(x_{1},-x_{2},-x_{3})\\\\
-\bw_{i}^{(3)}(x_{1},-x_{2},-x_{3})
\end{pmatrix},\quad i=1,6,
\end{align*}
and 
\begin{align*}
\bw_{i+6}(x_{1},x_{2},x_{3})=\begin{pmatrix}
\bw_{i+6}^{(1)}(x_{1},x_{2},x_{3})\\\\
\bw_{i+6}^{(2)}(x_{1},x_{2},x_{3})\\\\
\bw_{i+6}^{(3)}(x_{1},x_{2},x_{3})
\end{pmatrix}
=\begin{pmatrix}
-\bw_{i}^{(1)}(x_{1},-x_{2},-x_{3})\\\\
\bw_{i}^{(2)}(x_{1},-x_{2},-x_{3})\\\\
\bw_{i}^{(3)}(x_{1},-x_{2},-x_{3})
\end{pmatrix},\quad i=2,3,4,5.
\end{align*}
These relations, in combination with \eqref{def-Eij}, and \eqref{origin_sym0}--\eqref{origin_sym2}, yield the results in $(3)$. 
\end{proof}

The rest part in this section is devoted to the proof of $(i)$ and $(ii)$ in Proposition \ref{prop-asymp}. 

\subsection{Auxiliary results}
In this section, we give several auxiliary results that will be utilized in the proof of $(i)$ and $(ii)$ in Proposition \ref{prop-asymp}.

We shall adapt the auxiliary functions constructed in \cite{LX}. Introduce a scalar function $\bar{v}\in C^{2,\alpha}(\mathbb{R}^3)$, such that $\bar{v}=1$ on $\partial{D}_{1}$, $\bar{v}=0$ on $\partial{D}_{2}$, $\bar{v}=\Ocal\left(|\bx|^{-1}\right)$ as $|\bx|\rightarrow\infty$,
\begin{equation*}
\bar{v}(\bx)=\frac{x_{3}+\frac{\varepsilon}{2}-h_2(\bx')}{\delta(\bx')},\quad\hbox{in}\ \Omega_{2R_0},
\end{equation*}
and $\|\bar{v}\|_{C^{2,\alpha}(D^e\setminus\Omega_{R_0})}\leq\,C$, where $\Omega_{2R_0}$ and $\delta(\bx')$ are defined in \eqref{narrowreg} and \eqref{delta_x'}, respectively. Now we define $\bv_{i}\in C^{2,\alpha}(D^e)$ such that $\bv_{i}=\bw_{i}$ on $\partial D$, $\bv_{i}=\Ocal\left(|\bx|^{-1}\right)$ as $|\bx|\rightarrow\infty$, $\|\bv_{i}\|_{C^{2,\alpha}(D^e\setminus\Omega_{R_0})}\leq C$, and for $\bx\in\Omega_{2R_0}$,
\begin{equation}\label{auxiliary improved}
\begin{split}
\bv_{i}(\bx)&:=\bar{v}(\bx)\varkappa_i
+\frac{\lambda+\mu}{\lambda+2\mu}f(\bar{v}(\bx))\, \partial_{x_{i}}\delta(\bx')\,\varkappa_3,\qquad\qquad i=1,2,\\
\bv_{3}(\bx)&:=\bar{v}(\bx)\varkappa_3
+\frac{\lambda+\mu}{\mu}f(\bar{v}(\bx))\, \Big(\partial_{x_{1}}\delta(\bx')\,\varkappa_1+\partial_{x_{2}}\delta(\bx')\,\varkappa_2\Big),\\
\bv_{i}(\bx)&:=\bar{v}(\bx)\varkappa_i,\quad i=4,5,6,
\end{split}
\end{equation} 
where $\varkappa_i$ are defined in \eqref{eq:dpsi}, and
\begin{equation*}
f(\bar{v}):=\dfrac{1}{2}\Big(\bar{v}-\frac{1}{2}\Big)^{2}-\dfrac{1}{8}, ~f(\bar{v})=0 ~\mbox{on} ~\{|\bx'|<R_0\}.
\end{equation*} 
Similarly, we choose $\underline{v}\in C^{2,\alpha}(\mathbb{R}^3)$ such that $\underline{v}=1$ on $\partial{D}_{2}$, $\underline{v}=0$ on $\partial{D}_{1}$, $\underline{v}=\Ocal\left(|\bx|^{-1}\right)$ as $|\bx|\rightarrow\infty$, $\underline{v}=1-\bar{v}$ in $\Omega_{2R_0}$, and $\|\underline{v}\|_{C^{2,\alpha}(D^e\setminus\Omega_{R_0})}\leq\,C$. Then we take the auxiliary functions $\bv_{i}(\bx)$ by replacing $\bar{v}$ with $\underline{v}$ in \eqref{auxiliary improved}, $i=7,\dots,12$. 

When $\varepsilon=0$, denote by $D_1^0$ and $D_2^0$ the two hard inclusions, and set $D^{e,0}:=\mathbb{R}^{3} \backslash \overline{D_{1}^0 \cup D_{2}^0}$. 
Suppose $\bw_{i}^0$ satisfies
\begin{equation}\label{eq-wi0}
\begin{cases}
\mathcal{L}_{\lambda,\mu}\bw_{i}^0=0,&\mathrm{in}~D^{e,0},\\
\bw_{i}^0=\varkappa_i,&\mathrm{on}~\partial{D}_{1}^{0}\setminus\{0\},\\
\bw_{i}^0=0,&\mathrm{on}~\partial{D_{2}^{0}},\\
\bw_{i}^0=\Ocal\left(|\bx|^{-1}\right),& \mbox{as}~|\bx|\rightarrow\infty,
\end{cases}
\quad 
\begin{cases}
\mathcal{L}_{\lambda,\mu}\bw_{i+6}^0=0,&\mathrm{in}~D^{e,0},\\
\bw_{i+6}^0=\varkappa_i,&\mathrm{on}~\partial{D}_{2}^{0}\setminus\{0\},\\
\bw_{i+6}^0=0,&\mathrm{on}~\partial{D_{1}^{0}},\\
\bw_{i+6}^0=\Ocal\left(|\bx|^{-1}\right),& \mbox{as}~|\bx|\rightarrow\infty,
\end{cases}
\quad i=1,\dots,6.
\end{equation}	
The auxiliary functions $\bv_i^0$ associated with $\bw_i^0$ are constructed using the same way as described in \eqref{auxiliary improved} with $\varepsilon=0$.

We have the following gradient estimates (c.f. \cite[Theorems 3.2 and 3.3]{LX}).

\begin{lem}\label{lem-Dui}
Let $\bw_{i}\in{H}^1(D^e; \mathbb{R}^{3})$ be the weak solutions of \eqref{eq-wi}, $i=1,\dots,12$. Then for sufficiently small  $\varepsilon>0$, $i=1,\dots,12$, and $\bx\in\Omega_{R_0}$, we have 
\begin{align*}
\nabla \bw_{i}(\bx)=\nabla \bv_{i}(\bx)+\Ocal(1).
\end{align*}
\end{lem}

Denote 
\begin{equation*}
\Omega_r^0:=\left\{(\bx',x_{3})\in \mathbb{R}^{3}: h_{2}(\bx')<x_{3}<h_{1}(\bx'),~|\bx'|<r\right\},\quad 0<r\leq 2R_0.
\end{equation*}

Similar to Lemma \ref{lem-Dui}, we have the following result.
\begin{lem}\label{lem-Dui0}
Let $\bw_{i}^0\in{H}^1(D^{e,0}; \mathbb{R}^{3})$ be the weak solutions of \eqref{eq-wi0}, $i=1,\dots,12$. Then for  $\bx\in\Omega_{R_0}^0$, we have 
\begin{align*}
\nabla \bw_{i}^0(\bx)=\nabla \bv_{i}^0(\bx)+\Ocal(1).
\end{align*}
\end{lem}

Define
$$\mathcal{V}:=D^e\setminus\overline{D_{1}^{0}\cup D_{2}^{0}},$$
and
$$\mathcal{C}_{r}:=\left\{x\in\mathbb R^{3}: |\bx'|<r,~-\frac{\varepsilon}{2}+2\min_{|\bx'|=r}h_{2}(\bx')\leq x_{3}\leq\frac{\varepsilon}{2}+2\max_{|\bx'|=r}h_{1}(\bx')\right\},\quad r<R_0.$$
Then by employing the mean value theorem and Lemma \ref{lem-Dui}, we are able to obtain the convergence result for $\bw_{i}-\bw_{i}^{0}$. The specific details of the proof for this convergence result can be found in \cite[Lemma 5.1]{LX}. 

\begin{lem}\label{lem-differ}
Let $\bw_{i}$ and $\bw_{i}^{0}$ satisfy \eqref{eq-wi} and \eqref{eq-wi0}, respectively. Then we have
\begin{align}\label{differ1}
|(\bw_{i}-\bw_{i}^{0})(\bx)|\leq C\varepsilon^{1/2},\quad i=1,2,3,7,8,9,\quad \bx\in \mathcal{V}\setminus \mathcal{C}_{\varepsilon^{1/4}},
\end{align}
and
\begin{align*}
|(\bw_{i}-\bw_{i}^{0})(\bx)|\leq C\varepsilon^{2/3},\quad \alpha=4,5,6,10,11,12,\quad \bx\in \mathcal{V}\setminus \mathcal{C}_{\varepsilon^{1/3}},
\end{align*}
where $C$ is a constant independent of $\varepsilon$. 
\end{lem}

\subsection{Proof of $(i)$ in Proposition \ref{prop-asymp}}
With Lemmas \ref{lem-Dui}--\ref{lem-differ} in hand, we are ready to finish the proof of $(i)$ in Proposition \ref{prop-asymp}. 

\begin{proof}[\bf Proof of $(i)$ in Proposition \ref{prop-asymp}]
In view of \eqref{def-Eij}, let us divide $\mathcal{E}_{1,1}$ into three parts:
\begin{align*}
\mathcal{E}_{1,1}&=\int_{D^e}\left(\mathbb{C}e(\bw_1),e(\bw_1)\right)\\
&=\int_{D^e\setminus\Omega_{R_0}}\left(\mathbb{C}e(\bw_1),e(\bw_1)\right)+\int_{\Omega_{R_0}\setminus\Omega_{\varepsilon^{1/8}}}\left(\mathbb{C}e(\bw_1),e(\bw_1)\right)+\int_{\Omega_{\varepsilon^{1/8}}}\left(\mathbb{C}e(\bw_1),e(\bw_1)\right)\\
&=:\mbox{I}+\mbox{II}+\mbox{III}.
\end{align*}
Next we estimate $\mbox{I}$, $\mbox{II}$, and $\mbox{III}$ in three steps. 

{\bf Step 1.} We claim that
\begin{equation}\label{est-I}
\mbox{I}=\int_{D^{e,0}\setminus\Omega_{R_0}}\left(\mathbb{C}e(\bw_1^0),e(\bw_1^0)\right)+\Ocal(\varepsilon^{1/4}).
\end{equation}
Indeed, making use of $|(D_{1}^{0}\cup D_{2}^{0})\setminus(D_{1}\cup D_{2}\cup\Omega_{R_0})|, |(D_{1}\cup D_2)\setminus (D_{1}^{0}\cup D_{2}^{0})|\leq C\varepsilon$, $|e(\bw_1)|$ and $|e(\bw_1^0)|$ are bounded in $(D_{1}^{0}\cup D_{2}^{0})\setminus(D_{1}\cup D_{2}\cup\Omega_{R_0})$ and $(D_{1}\cup D_2)\setminus (D_{1}^{0}\cup D_{2}^{0})$, respectively, we have 
\begin{align}\label{est-I1}
\mbox{I}&=\int_{D^e\setminus(D_1^0\cup D_{2}^{0}\cup\Omega_{R_0})}\left(\Big(\mathbb{C}e(\bw_{1}),e(\bw_{1})\Big)-\Big(\mathbb{C}e(\bw_{1}^{0}),e(\bw_{1}^{0})\Big)\right)\nonumber\\
&\quad+\int_{(D_{1}^{0}\cup D_{2}^{0})\setminus(D_{1}\cup D_{2}\cup\Omega_{R_0})}\Big(\mathbb{C}e(\bw_{1}),e(\bw_{1})\Big)-\int_{(D_{1}\cup D_2)\setminus (D_{1}^{0}\cup D_{2}^{0})}\Big(\mathbb{C}e(\bw_{1}^{0}),e(\bw_{1}^{0})\Big)\nonumber\\
&=\int_{D^e\setminus(D_1^0\cup D_{2}^{0}\cup\Omega_{R_0})}\Big(\mathbb{C}e(\bw_{1}-\bw_{1}^{0}),e(\bw_{1}+\bw_{1}^{0})\Big)+\Ocal(\varepsilon).
\end{align}
Denote 
$$M:=2\max_{\bx,\by\in\overline{D_1^0\cup D_2^0}}|\bx-\by|.$$
Then we have $D_1\cup D_2\subset B_{M}$. In view of 
$$\frac{\partial\bw_{1}}{\partial r}=\Ocal\left(|\bx|^{-2}\right),\quad \frac{\partial\bw_{1}^0}{\partial r}=\Ocal\left(|\bx|^{-2}\right)\quad\mbox{as}~|\bx|\rightarrow\infty,$$
using the integration by parts, \eqref{differ1}, and $|\nabla(\bw_{1}+\bw_{1}^{0})|_{L^\infty(\mathcal{V}\setminus{\Omega_{R_0}})}\leq C$, we have 
\begin{align*}
\left|\int_{\mathbb R^3\setminus B_M}\Big(\mathbb{C}e(\bw_{1}-\bw_{1}^{0}),e(\bw_{1}+\bw_{1}^{0})\Big)\right|&=\left|\int_{\mathbb R^3\setminus B_M}\Big(\mathbb{C}e(\bw_{1}+\bw_{1}^{0}),e(\bw_{1}-\bw_{1}^{0})\Big)\right|\\
&=\left|\int_{\partial B_M}\frac{\partial(\bw_{1}+\bw_{1}^{0})}{\partial{\bnu}}\big|_{+}(\bw_{1}-\bw_{1}^{0})\right|\leq C\varepsilon^{1/2}.
\end{align*}
Thus, \eqref{est-I1} becomes
\begin{align}\label{est-I1-1}
\mbox{I}=\int_{B_M\setminus(D_1\cup D_2\cup D_1^0\cup D_{2}^{0}\cup\Omega_{R_0})}\Big(\mathbb{C}e(\bw_{1}-\bw_{1}^{0}),e(\bw_{1}+\bw_{1}^{0})\Big)+\Ocal(\varepsilon^{1/2}).
\end{align}
Recalling
$$\mathcal{L}_{\lambda,\mu}\bw_{1}=0,\quad \mbox{in}~  D^e\setminus (D_{1}^{0}\cup D_{2}^{0}\cup\Omega_{R_0}).$$
For any fixed $\bx_0\in D^e\setminus (D_{1}^{0}\cup D_{2}^{0}\cup\Omega_{R_0})$, making use of the condition $\bw_{1}=\Ocal\left(|\bx|^{-1}\right)$ as $|\bx|\rightarrow\infty$, we find that for any $\varepsilon_0\in(0,1)$, there exists a constant $M_0:=M_0(\varepsilon_0)\geq|\bx_0|$ such that 
$$|\bw_{1}|_{\partial B_{M_0}}<\varepsilon_0.$$
Together with the maximum modulus estimate for Lam\'{e} systems in $B_{M_0}\setminus (D_{1}\cup D_{2}\cup D_{1}^{0}\cup D_{2}^{0}\cup\Omega_{R_0})$, we have 
$$|\bw_{1}|\leq 1,\quad \mbox{in}~ D^e\setminus (D_{1}^{0}\cup D_{2}^{0}\cup\Omega_{R_0}).$$
Similarly, we have 
$$|\bw_{1}^{0}|\leq1,\quad \mbox{in}~ D^e\setminus (D_{1}^{0}\cup D_{2}^{0}\cup\Omega_{R_0}).$$
Since $\partial D_{1}, \partial D_{1}^{0}, \partial D_{2}$, and $\partial D_2^0$ are $C^{2,\alpha}$, we have
\begin{equation}\label{DD v1 1}
|\nabla^{2}(\bw_{1}-\bw_{1}^{0})|\leq|\nabla^{2}\bw_{1}|+|\nabla^{2}\bw_{1}^{0}|\leq C,\quad\mbox{in}~D^e\setminus(D_{1}^{0}\cup D_{2}^{0}\cup\Omega_{R_0}).
\end{equation}
Note that 
$$\mathcal{L}_{\lambda,\mu}(\bw_{1}-\bw_{1}^{0})=0,\quad\mbox{in}~ D^e\setminus (D_{1}^{0}\cup D_{2}^{0}\cup\Omega_{R_0}).$$
Then it follows from \eqref{differ1} that
\begin{equation}\label{es v11 v1*1}
\|\bw_{1}-\bw_{1}^{0}\|_{L^{\infty}(D^e\setminus(D_{1}^{0}\cup D_{2}^{0}\cup\Omega_{\varepsilon^{1/4}}))}\leq C\varepsilon^{1/2}.
\end{equation}
By using the interpolation inequality, \eqref{DD v1 1}, and \eqref{es v11 v1*1}, we obtain for $\bx\in B_{M}\setminus(D_{1}^{0}\cup D_{2}^{0}\cup\Omega_{R_0})$,
\begin{equation}\label{D v1 *1}
|\nabla(\bw_{1}-\bw_{1}^{0})(\bx)|\leq \varepsilon^{1/4}|\nabla^{2}(\bw_{1}-\bw_{1}^{0})(\bx)|+C\varepsilon^{-1/4}|(\bw_{1}-\bw_{1}^{0})(\bx)|\leq C\varepsilon^{1/4}.
\end{equation}
Substituting \eqref{D v1 *1} into \eqref{est-I1-1}, we obtain \eqref{est-I}.

{\bf Step 2.} Proof of 
\begin{equation}\label{est-II}
\mbox{II}=\int_{\Omega_{R_0}^{0}\setminus\Omega_{\varepsilon^{1/8}}^{0}}\Big(\mathbb{C}e(\bv_{1}^{0}),e(\bv_{1}^{0})\Big)+\mathcal{C}_1^1+\Ocal(\varepsilon^{1/4}|\log\varepsilon|),
\end{equation}
where 
\begin{align*}
\mathcal{C}_1^1=2\int_{\Omega_{R_0}^{0}}\Big(\mathbb{C}e(\bv_{1}^{0}),e(\bw_{1}^{0}-\bv_{1}^{0})\Big)+\int_{\Omega_{R_0}^{0}}\Big(\mathbb{C}e(\bw_{1}^{0}-\bv_{1}^{0}),e(\bw_{1}^{0}-\bv_{1}^{0})\Big)
\end{align*}
is a constant independent of $\varepsilon$. 
Indeed, 
\begin{align*}
&\mbox{II}-\int_{\Omega_{R_0}^0\setminus\Omega_{\varepsilon^{1/8}}^0}\left(\mathbb{C}e(\bw_1^0),e(\bw_1^0)\right)\\
&=\int_{(\Omega_{R_0}\setminus\Omega_{\varepsilon^{1/8}})\setminus(\Omega_{R_0}^{0}\setminus\Omega_{\varepsilon^{1/8}}^{0})}\Big(\mathbb{C}e(\bw_{1}),e(\bw_{1})\Big)\nonumber\\
&\quad+2\int_{\Omega_{R_0}^{0}\setminus\Omega_{\varepsilon^{1/8}}^{0}}\Big(\mathbb{C}e(\bw_{1}^{0}),e(\bw_{1}^{0}-\bw_{1}^{0})\Big)+\int_{\Omega_{R_0}^{0}\setminus\Omega_{\varepsilon^{1/8}}^{0}}\Big(\mathbb{C}e(\bw_{1}-\bw_{1}^{0}),e(\bw_{1}-\bw_{1}^{0})\Big).
\end{align*}
Set
\begin{align*}
\begin{cases}
\bx'-\bz'=|\bz'|^{2}\by',\\
x_{3}=|\bz'|^{2}y_{3}.
\end{cases}
\end{align*}
Then we perform a rescaling transformation to map the region $\Omega_{|\bz'|+|\bz'|^{2}}^0\setminus\Omega_{|\bz'|}^0$ onto $Q_{1}^{0}:=\left\{\by\in\mathbb{R}^{3}: \frac{1}{\delta}h_{2}(\delta\,\by'+\bz')<y_{3}
<\frac{1}{\delta}h_{1}(\delta\,\by'+\bz'),~|\by'|<1\right\}$. Note that $Q_{1}^{0}$ is a nearly unit-size square  whose height is uniformly bounded from above and below by a constant independent of $\varepsilon$. Set 
$$Q_{1}:=\left\{\by\in\mathbb{R}^{3}: -\frac{\varepsilon}{2\delta}+\frac{1}{\delta}h_{2}(\delta\,\by'+\bz')<y_{3}
<\frac{\varepsilon}{2\delta}+\frac{1}{\delta}h_{1}(\delta\,\by'+\bz'),~|\by'|<1\right\}.$$
Let
$$\mathbb W_{1}=\bw_{1}(\bz'+|\bz'|^{2}\by',|\bz'|^{2}y_{3})\quad\mbox{in}~Q_{1},\quad\mbox{and}\quad \mathbb W_{1}^{0}=\bw_{1}^{0}(\bz'+|\bz'|^{2}\by',|\bz'|^{2}y_{3})\quad\mbox{in}~Q_{1}^{0}.$$
Then as in \eqref{D v1 *1} and using the scaling, we get
\begin{equation}\label{difference Dv1 1}
|\nabla(\bw_{1}-\bw_{1}^{0})|\leq C\varepsilon^{1/4}|\bx'|^{-2}\quad\mbox{in}~\Omega_{R_0}^{0}\setminus\Omega_{\varepsilon^{1/8}}^{0}.
\end{equation}
Similarly, we have 
\begin{equation}\label{Dv11}
|\nabla \bw_{1}|\leq C|\bx'|^{-2}\quad\mbox{in}~\Omega_{R_0}\setminus\Omega_{\varepsilon^{1/8}},\quad |\nabla \bw_{1}^{0}|\leq C|\bx'|^{-2}\quad\mbox{in}~\Omega_{R}^{0}\setminus\Omega_{\varepsilon^{1/8}}^{0}.
\end{equation}
It follows from $|(\Omega_{R_0}\setminus\Omega_{\varepsilon^{1/8}})\setminus(\Omega_{R_0}^{0}\setminus\Omega_{\varepsilon^{1/8}}^{0})|\leq C\varepsilon$ and \eqref{Dv11} that
\begin{align*}
\left|\int_{(\Omega_{R_0}\setminus\Omega_{\varepsilon^{1/8}})\setminus(\Omega_{R_0}^{0}\setminus\Omega_{\varepsilon^{1/8}}^{0})}\Big(\mathbb{C}e(\bw_{1}),e(\bw_{1})\Big)\right|\leq C\int_{\varepsilon^{1/8}<|\bx'|\leq R_0}\frac{\varepsilon}{|\bx'|^{4}}d\bx'\leq C\varepsilon^{3/4}.
\end{align*}
Using \eqref{difference Dv1 1} and \eqref{Dv11}, we have 
\begin{align*}
&\left|2\int_{\Omega_{R_0}^{0}\setminus\Omega_{\varepsilon^{1/8}}^{0}}\Big(\mathbb{C}e(\bw_{1}^{0}),e(\bw_{1}^{0}-\bw_{1}^{0})\Big)+\int_{\Omega_{R_0}^{0}\setminus\Omega_{\varepsilon^{1/8}}^{0}}\Big(\mathbb{C}e(\bw_{1}-\bw_{1}^{0}),e(\bw_{1}-\bw_{1}^{0})\Big)\right|\\
&\leq C\int_{\varepsilon^{1/8}<|\bx'|\leq R_0}\frac{\varepsilon^{1/4}}{|\bx'|^{2}}d\bx'+C\int_{\varepsilon^{1/8}<|\bx'|\leq R_0}\frac{\varepsilon^{1/2}}{|\bx'|^{2}}d\bx'\leq C\varepsilon^{1/4}|\log\varepsilon|.
\end{align*}
Hence, we obtain 
\begin{equation}\label{est-II-1}
\mbox{II}-\int_{\Omega_{R_0}^0\setminus\Omega_{\varepsilon^{1/8}}^0}\left(\mathbb{C}e(\bw_1^0),e(\bw_1^0)\right)=\Ocal(\varepsilon^{1/4}|\log\varepsilon|).
\end{equation}
Recalling the definition of $\bv_1^0$ as follows:
\begin{equation}\label{def-v10}
\bv_{1}^0(\bx)=\bar{v}^0(\bx)\varkappa_1
+\frac{\lambda+\mu}{\lambda+2\mu}f(\bar{v}^0(\bx))\, \partial_{x_{1}}\big(h_1(\bx')-h_2(\bx')\big)\,\varkappa_3,
\end{equation}
where $\bar{v}^0\in C^{2,\alpha}(\mathbb{R}^3)$ satisfying $\bar{v}^0=1$ on $\partial{D}_{1}^0\setminus\{0\}$, $\bar{v}^0=0$ on $\partial{D}_{2}^0$, $\bar{v}^0=\Ocal\left(|\bx|^{-1}\right)$ as $|\bx|\rightarrow\infty$,
\begin{equation*}
\bar{v}^0(\bx)=\frac{x_{3}-h_2(\bx')}{h_1(\bx')-h_2(\bx')},\quad\hbox{in}\ \Omega_{2R_0}^0,
\end{equation*}
and $\|\bar{v}^0\|_{C^{2,\alpha}(D^{e,0}\setminus\Omega_{R_0}^0)}\leq\,C$. Then together with Lemma \ref{lem-Dui0}, we obtain 
\begin{align*}
&\int_{\Omega_{R_0}^0\setminus\Omega_{\varepsilon^{1/8}}^0}\left(\mathbb{C}e(\bw_1^0),e(\bw_1^0)\right)\\
&=\int_{\Omega_{R_0}^{0}\setminus\Omega_{\varepsilon^{1/8}}^{0}}\Big(\mathbb{C}e(\bv_{1}^{0}),e(\bv_{1}^{0})\Big)+2\int_{\Omega_{R_0}^{0}\setminus\Omega_{\varepsilon^{1/8}}^{0}}\Big(\mathbb{C}e(\bv_{1}^{0}),e(\bw_{1}^{0}-\bv_{1}^{0})\Big)\nonumber\\ &\quad+\int_{\Omega_{R_0}^{0}\setminus\Omega_{\varepsilon^{1/8}}^{0}}\Big(\mathbb{C}e(\bw_{1}^{0}-\bv_{1}^{0}),e(\bw_{1}^{0}-\bv_{1}^{0})\Big)\nonumber\\
&=\int_{\Omega_{R_0}^{0}\setminus\Omega_{\varepsilon^{1/8}}^{0}}\Big(\mathbb{C}e(\bv_{1}^{0}),e(\bv_{1}^{0})\Big)+\mathcal{C}_1^1+\Ocal(\varepsilon^{1/4}).
\end{align*}
Coming back to \eqref{est-II-1}, we derive \eqref{est-II}.

{\bf Step 3.} We are now in a position to prove Proposition \ref{prop-asymp} for $i=1$. From Lemma \ref{lem-Dui} and $|\nabla\bv_1|\leq\frac{C}{\delta(\bx')}$, it follows that 
\begin{align*}
\mbox{III}
&=\int_{\Omega_{\varepsilon^{1/8}}}\Big(\mathbb{C}e(\bv_{1}),e(\bv_{1})\Big)+2\int_{\Omega_{\varepsilon^{1/8}}}\Big(\mathbb{C}e(\bv_{1}),e(\bw_{1}-\bv_{1})\Big)\nonumber\\
&\quad+\int_{\Omega_{\varepsilon^{1/8}}}\Big(\mathbb{C}e(\bw_{1}-\bv_{1}),e(\bw_{1}-\bv_{1})\Big)\nonumber\\
&=\int_{\Omega_{\varepsilon^{1/8}}}\Big(\mathbb{C}e(\bv_{1}),e(\bv_{1})\Big)+\Ocal(\varepsilon^{1/4}).
\end{align*}
This in combination with \eqref{est-I} and \eqref{est-II} yields
\begin{align}\label{E11-est}
\mathcal{E}_{1,1}=\int_{\Omega_{\varepsilon^{1/8}}}\Big(\mathbb{C}e(\bv_{1}),e(\bv_{1})\Big)+\int_{\Omega_{R_0}^{0}\setminus\Omega_{\varepsilon^{1/8}}^{0}}\Big(\mathbb{C}e(\bv_{1}^{0}),e(\bv_{1}^{0})\Big)+\mathcal{C}_1^2+\Ocal(\varepsilon^{1/4}|\log\varepsilon|),
\end{align}
where 
\begin{equation*}
\mathcal{C}_1^2=\mathcal{C}_1^1+\int_{D^{e,0}\setminus\Omega_{R_0}}\left(\mathbb{C}e(\bw_1^0),e(\bw_1^0)\right).
\end{equation*}
Recalling the definition of $\bv_1$ in \eqref{auxiliary improved} and using \eqref{delta_x'}, we find that the biggest term in $\Big(\mathbb{C}e(\bv_{1}),e(\bv_{1})\Big)$ is 
\begin{equation*}
\mu\Big(\partial_{x_3}\bv_1^{(1)}\Big)^2=\frac{\mu}{\delta^2(\bx')}=\frac{\mu}{\Big(\varepsilon+\kappa|\bx'|^{2}+O(|\bx'|^{2+\alpha})\Big)^2}.
\end{equation*}
The remaining terms are bounded by $\frac{C|\bx'|}{\delta^2(\bx')}$. Similarly, the leading order terms in $\Big(\mathbb{C}e(\bv_{1}^{0}),e(\bv_{1}^{0})\Big)$ is 
\begin{equation*}
\mu\Big(\partial_{x_3}(\bv_1^0)^{(1)}\Big)^2=\frac{\mu}{\Big(\kappa|\bx'|^{2}+O(|\bx'|^{2+\alpha})\Big)^2}.
\end{equation*}
These imply that
\begin{align}\label{equality e11}
\mathcal{E}_{1,1}&=\int_{\Omega_{\varepsilon^{1/8}}}\frac{\mu}{\Big(\varepsilon+\kappa|\bx'|^{2}+O(|\bx'|^{2+\alpha})\Big)^2} d\bx+\int_{\Omega_{R_0}^{0}\setminus\Omega_{\varepsilon^{1/8}}^{0}}\frac{\mu}{\Big(\kappa|\bx'|^{2}+O(|\bx'|^{2+\alpha})\Big)^2} d\bx\nonumber\\
&\quad+\mathcal{C}_1^2+\Ocal(\varepsilon^{1/4}|\log\varepsilon|)\nonumber\\
&=\int_{|\bx'|\leq\varepsilon^{1/8}}\frac{\mu}{\varepsilon+\kappa|\bx'|^{2}+O(|\bx'|^{2+\alpha})} d\bx'+\int_{\varepsilon^{1/8}<|\bx'|\leq R_0}\frac{\mu}{\kappa|\bx'|^{2}+O(|\bx'|^{2+\alpha})} d\bx'\nonumber\\
&\quad+\mathcal{C}_1^2+\Ocal(\varepsilon^{1/4}|\log\varepsilon|)\nonumber\\
&=\frac{\mu\,\pi}{\kappa}|\log\varepsilon|+\mathcal{C}_{1,1}+o(1),
\end{align}
where $\mathcal{C}_{1,1}$ is a constant independent of $\varepsilon$. Therefore, we finish the proof of the estimate of $\mathcal{E}_{1,1}$. The cases of $i=2,3,7,8,9$ follow a similar approach and are omitted. 

For $(i,j)=(1,7)$, replicating the same argument as above up to \eqref{E11-est}, we have 
\begin{align}\label{edt-E17}
\mathcal{E}_{1,7}=\int_{\Omega_{\varepsilon^{1/8}}}\Big(\mathbb{C}e(\bv_{1}),e(\bv_{7})\Big)+\int_{\Omega_{R_0}^{0}\setminus\Omega_{\varepsilon^{1/8}}^{0}}\Big(\mathbb{C}e(\bv_{1}^{0}),e(\bv_{7}^{0})\Big)+\mathcal{C}+\Ocal(\varepsilon^{1/4}|\log\varepsilon|),
\end{align}
where $\mathcal{C}$ is a constant independent of $\varepsilon$, $\bv_{7}$ and $\bv_{7}^{0}$ are defined by 
\begin{equation*}
\bv_{7}=(1-\bar{v})\varkappa_1
+\frac{\lambda+\mu}{\lambda+2\mu}f(1-\bar{v})\, \partial_{x_{1}}\delta(\bx')\,\varkappa_3\quad\mbox{in}~\Omega_{2R_0}
\end{equation*}
and 
\begin{equation*}
\bv_{7}^0=(1-\bar{v}^0)\varkappa_1
+\frac{\lambda+\mu}{\lambda+2\mu}f(1-\bar{v}^0)\, \partial_{x_{1}}\big(h_1(\bx')-h_2(\bx')\big)\,\varkappa_3\quad\mbox{in}~\Omega_{2R_0}^0.
\end{equation*}
Together with the definition of $\bv_1$ in \eqref{auxiliary improved}, \eqref{def-v10}, and using \eqref{delta_x'}, we get that the dominate terms in $\Big(\mathbb{C}e(\bv_{1}),e(\bv_{7})\Big)$ and $\Big(\mathbb{C}e(\bv_{1}^{0}),e(\bv_{1}^{0})\Big)$ are, respectively,
\begin{equation*}
\mu\partial_{x_3}\bv_1^{(1)}\partial_{x_3}\bv_7^{(1)}=-\frac{\mu}{\delta^2(\bx')}=-\frac{\mu}{\Big(\varepsilon+\kappa|\bx'|^{2}+O(|\bx'|^{2+\alpha})\Big)^2}
\end{equation*}
and
\begin{equation*}
\mu\partial_{x_3}(\bv_1^0)^{(1)}\partial_{x_3}(\bv_7^0)^{(1)}=-\frac{\mu}{\Big(\kappa|\bx'|^{2}+O(|\bx'|^{2+\alpha})\Big)^2}.
\end{equation*}
Then coming back to \eqref{edt-E17} and following the same calculation as in \eqref{equality e11}, we derive the estimate of $\mathcal{E}_{1,7}$. We omit the proof of the remaining cases since the argument is similar. Thus, the estimates in $(i)$ are proved.
\end{proof}

\subsection{Proof of $(ii)$ and $(iv)$ in Proposition \ref{prop-asymp}}
\begin{proof}[\bf Proof of $(ii)$ in Proposition \ref{prop-asymp}]
The relations \eqref{sym-Ei6} and \eqref{sym-Ei} directly follow from the definition of $\mathcal{E}_{i,j}$ in \eqref{def-Eij} and $(\mathbb{C}A,B)=(A,\mathbb{C}B)$ for every pair of $3\times 3$ matrices $A$ and $B$.

It follows from \eqref{def-Eij}, Lemma \ref{lem-Dui}, and \eqref{auxiliary improved} that
\begin{equation*}
\mathcal{E}_{4,4}=\int_{D^e}\left(\mathbb{C}e(\bw_4),e(\bw_4)\right)\leq C\int_{\Omega_{R_0}}|\nabla\bv_4|^2+C\leq C\int_{\Omega_{R_0}}\frac{|\bx'|^2}{\delta^2(\bx')}\ d\bx+C\leq C.
\end{equation*}
Furthermore, by using \cite[(4.18)]{BLL2015}, we obtain
\begin{align*}
\mathcal{E}_{4,4}\geq\frac{1}{C}\int_{\Omega_{R_0}\setminus\Omega_{R_0/2}}|e(\bw_4)|^2&\geq \frac{1}{C}\int_{\Omega_{R_0}\setminus\Omega_{R_0/2}}|\nabla\bw_4|^2\\
&\geq \frac{1}{C}\int_{\Omega_{R_0}\setminus\Omega_{R_0/2}}|\partial_{x_3}\bw_4^{(1)}|^2\\
&\geq\frac{1}{C}\int_{\Omega_{R_0}\setminus\Omega_{R_0/2}}\frac{x_2^2}{\delta^2(\bx')}\ d\bx\geq \frac{1}{C}.
\end{align*}
The proof of the case $i=5,6,10,11,12$ is similar and thus is omitted.
\end{proof}

\begin{proof}[\bf Proof of $(iv)$ in Proposition \ref{prop-asymp}]
Since $\mathcal{E}_{i,i}=\mathcal{E}_{i+6,i+6}$, $i=1,\dots,6$, we have 
\begin{equation*}
\mathcal{E}_{i,i}+\mathcal{E}_{i,i+6}=\frac{1}{2}\int_{D^e}\left(\mathbb{C}e(\bw_i+\bw_{i+6}),e(\bw_{i}+\bw_{i+6})\right).
\end{equation*}
On one hand, recalling that $\bv_{i+6}$ is defined as in \eqref{auxiliary improved} by replacing $\bar v$ with $\underline v$, then by using Lemma \ref{lem-Dui} and the definition of $\bv_{i}$ in \eqref{auxiliary improved}, we obtain
\begin{equation*}
\mathcal{E}_{i,i}+\mathcal{E}_{i,i+6}\leq C.
\end{equation*}
On the other hand, it follows from \cite[(4.18)]{BLL2015} that
\begin{equation*}
\mathcal{E}_{i,i}+\mathcal{E}_{i,i+6}\geq \frac{1}{C}\int_{D^e}|e(\bw_i+\bw_{i+6})|^2\ d\bx\geq0.
\end{equation*}
By \eqref{sym-Ei6} and $\mathcal{E}_{i,i}=\mathcal{E}_{i+6,i+6}$, $i=1,\dots,6$, we obtain
\begin{equation*}
\mathcal{E}_{i,i}-\mathcal{E}_{i,i+6}=\frac{1}{2}\int_{D^e}\left(\mathbb{C}e(\bw_i-\bw_{i+6}),e(\bw_{i}-\bw_{i+6})\right).
\end{equation*}
Then the same argument as above yields
\begin{equation*}
0\leq \mathcal{E}_{i,i}-\mathcal{E}_{i,i+6}\leq C.
\end{equation*}
Together with $\mathcal{E}_{i,i}=\mathcal{E}_{i+6,i+6}$ and $\mathcal{E}_{i,i+6}=\mathcal{E}_{i+6,i}$, the second inequality in \eqref{sum-Eii} is proved. The proof of Proposition \ref{prop-asymp} is finished.
\end{proof}





\bibliographystyle{abbrv}
\bibliography{res_stre}{}

\end{document}